\documentclass[11pt]{amsart}

\usepackage{amsmath}
\usepackage{amssymb}
\usepackage{amsthm}
\usepackage{graphicx} 
\usepackage{url}
\usepackage[bookmarks]{hyperref}
\usepackage[msc-links]{amsrefs} 
\usepackage{multirow}
\usepackage{array}
\usepackage{multirow}
\usepackage{wasysym}
\usepackage[shortlabels]{enumitem}
\usepackage{color}
\usepackage[normalem]{ulem}
\usepackage{soul}
\usepackage{rotate}

\theoremstyle{plain}
\newtheorem{introtheorem}{Theorem}

\newtheorem{introcoro}[introtheorem]{Corollary}

\theoremstyle{plain}
\newtheorem{theorem}{Theorem}[section]

\theoremstyle{definition}
\newtheorem{definition}[theorem]{Definition}
\newtheorem{notation}[theorem]{Notation}

\newtheorem{question}[theorem]{Question}
\newtheorem{remark}[theorem]{Remark}

\newcommand{\term}[1]{{\bf #1}}

\newcommand{\bord}{\partial}
\newcommand{\bsm}{\begin{smallmatrix}}
\newcommand{\esm}{\end{smallmatrix}}

\newcommand{\fgeod}{\Phi_\mathrm{geod}}

\newcommand{\Fs}{\mathcal{F}^s}
\newcommand{\Fu}{\mathcal{F}^u}

\renewcommand{\ge}{\geqslant}

\newcommand{\Hy}{\mathbb{H}^2}

\renewcommand{\le}{\leqslant}

\newcommand{\OO}{\mathcal{O}}

\newcommand{\bpm}{\begin{pmatrix}}
\newcommand{\epm}{\end{pmatrix}}

\newcommand{\HH}{\mathbb{H}}
\newcommand{\RR}{\mathbb{R}}

\newcommand{\DD}{\mathbb{D}}

\newcommand{\SLZ}{\mathrm{SL}_2(\ZZ)}

\newcommand{\slk}{\mathrm{slk}}

\newcommand{\TT}{\mathbb{T}}

\newcommand{\U}{\mathrm{T}^1}

\renewcommand{\vec}[1]{\overset\rightarrow{#1}}
\newcommand{\revec}[1]{\overset\leftarrow{#1}}
\newcommand{\vrevec}[1]{\overset\leftrightarrow{#1}}

\newcommand{\ZZ}{\mathbb{Z}}

\title{Genus one Birkhoff sections for geodesic flows on orbifolds}

\author{Pierre Dehornoy}
\address{Université Aix-Marseille, CNRS, I2M, 13000 Marseille, France}
\email{pierre.dehornoy@univ-amu.fr}
\urladdr{https://www.i2m.univ-amu.fr/perso/pierre.dehornoy/}
\thanks{Supported by the ANR projects Groméov ANR-19-CE40-0007 and AnoDyn ANR-24-CE40-5065}

\date{1st version oct. 2019, current version march. 2026}

\begin{document}

\begin{abstract}
For $\OO$ a hyperbolic orientable 2-orbifold of genus~$g$ with at most $2g+6$ conic points, we prove that the geodesic flow on~$\U\OO$ admits a Birkhoff section whose genus is one.  
Together with a result of Minakawa, this implies that this flow is almost equivalent to the suspension flow of the $(\begin{smallmatrix}2&1\\1&1\end{smallmatrix})$-map on the torus. 
\end{abstract}

\maketitle


\section{Introduction}
\label{S:Introduction}

\subsection{Main statement}

Given a flow~$(\Phi^t)_{t\in\RR}$ on a compact 3-manifold~$M$, a \term{Birkhoff section} 
is an immersion~$\iota:S\to M$ of an oriented surface with boundary~$S$ such that: 
\begin{enumerate}[(i)]
\item the interior of~$\iota(S)$ is embedded in~$M$ and positively transverse to~$\Phi$,
\item the boundary~$\iota(\bord S)$ is immersed in~$M$ and tangent to the flow~$\Phi$.
\item $S$ intersects all orbits in bounded time, \emph{i.e.}, $\exists T>0$ such that $\Phi^{[0,T]}(\iota(S))=M$.
\end{enumerate}

It follows from the definition that the image of $\partial S$ under the immersion $\iota$ is a union ${\bigcup_{j=1}^c\gamma_j}$ of finitely many periodic orbits of the flow, and for each component $\alpha$ of $\partial S$ the restriction $\iota:\alpha\to\gamma_{j(\alpha)}$ is a covering map. 
In this paper, all constructed Birkhoff sections turn out to be {\it embedded}, meaning that these covering maps are in fact homeomorphisms. 
In general we omit the immersion~$\iota$ and regard $S$ directly as a subset in~$M$. 

A 2-dimensional {orbifold}, or \term{2-orbifold}, is a metric space locally modeled on the quotient of a Riemannian surface by a discrete group of isometries. 
In this note, we focus on \term{orientable hyperbolic 2-orbifolds} which can be described as quotients of the hyperbolic plane~$\Hy$ by discrete cocompact groups of orientation-preserving isometries~\cite{ThurstonNotes}. 
Such a 2-orbifold~$\OO$ can then be seen as a topological compact surface~$\Sigma_\OO$ which has a genus~$g$, endowed with a hyperbolic metric, and with finitely many singular points~$a_1, \dots, a_n$, called~{\it conic points}, where it is locally isometric to the quotient~$\mathbb{D}_{p_i}$ of a hyperbolic disc by a group of order~$p_i$ rotations. 
The type of this orientable hyperbolic 2-orbifold is the tuple $(g; p_1, \dots, p_n)$. 
The 2-orbifold is~\term{of small type} if either $g=0$, or $g\ge 1$ and $n\le 2g+6$. 

To every Riemannian 2-orbifold~$\OO$ is associated its \term{unitary tangent bundle}~$\U\OO$ which is a Seifert fibered 3-manifold. 
It admits a \term{geodesic flow}~$\fgeod$ whose orbits are lifts of geodesics on~$\OO$~\cite{Hadamard}. 
When $\OO$ is hyperbolic, $\fgeod$ is an Anosov flow, therefore it is chaotic and structurally stable~\cite{Anosov}. 

The goal of this paper is to prove

\begin{introtheorem}\label{T:G1BS}
If $\OO$ is an orientable hyperbolic 2-orbifold of small type, then the geodesic flow on~$\U\OO$ admits a genus-one Birkhoff section.
\end{introtheorem}

Theorem~\ref{T:G1BS} generalizes Fried's seminal result on geodesic flows on hyperbolic surfaces~\cite{FriedAnosov}, as well as several more recent constructions~\cites{GenusOne, HM}. 

\begin{remark}
In the first version of a paper we wrote\footnote{\url{https://arxiv.org/abs/1910.08457v1}}, we claimed that the geodesic flow on {\it any} hyperbolic orbifold admits a genus one Birkhoff section. 
Unfortunately, during the revision process, we realized that, in the most general case, the Birkhoff sections we constructed have genus two. 
As of this writing we cannot bridge the gap, hence our restriction here to small orbifolds (and the need for a different construction when $g>0$). 
We still believe the general result holds, but it probably requires more intricate constructions. 
In particular all the Birkhoff sections we construct here (as well as those arising in the literature so far) are \emph{negative} in the sense that the orientation of every boundary component inherited from the coorientation of the surface by the flow is \emph{opposite} to the orientation induced by the flow. 
It seems likely that, for some orbifolds not of small type, one must consider \emph{mixed} Birkhoff sections, {\it i.e.}, sections with both positive and negative boundary components.  
We therefore ask

\begin{question}
If $\OO$ is a hyperbolic orientable 2-orbifold with positive genus~$g$ with $n>2g+6$ conic points, does $\OO$ admit a genus one Birkhoff section with all negative boundary components? 
\end{question}

Note that a recent work of Tsang~\cite{Tsang} may imply the existence of genus 1 Birkhoff sections for all geodesic flows on orientable hyperbolic 2-orbifolds, without describing them explicitly. 
\end{remark}

\subsection{Motivation and corollaries}
Two flows $\Phi_1$ and $\Phi_2$ on a manifold~$M$ are said to be {\it orbitally equivalent} if there is a homeomorphism of~$M$ that maps oriented orbits of $\Phi_1$ onto oriented orbits of~$\Phi_2$. 
In other words, the 1-dimensional foliations whose leaves are the orbits are conjugated. 

In general, orbit equivalence classes of flows on a given manifold may be numerous (even uncountable) and complicated. 
The situation is different if one restricts to {\it Anosov flows}~\cite{Anosov}. 
Indeed, because of their structural stability ---the perturbation of an Anosov flow is another Anosov flow that is orbitally equivalent to the initial flow--- orbit equivalence classes are open, and there is at most a countable number of them. 

There are two classical families of Anosov flows, namely the suspension flows of hyperbolic matrices in~$\SLZ$, in which case the underlying manifold is a torus bundle over the circle, and geodesic flows on unitary tangent bundles of hyperbolic 2-dimensional orientable orbifolds. 
For each of the underlying manifolds, there is only one orbit equivalence class of Anosov flow. 

Many other Anosov flows can be built, in particular using different types of surgery constructions. 
One of the oldest such surgeries in known as Goodman-Fried surgery~\cite{Goodman, FriedAnosov}. 
One way of presenting it consists in starting with an Anosov flow~$\Phi$ on a 3-manifold~$M$, choosing a finite number of periodic orbits~$\gamma_1, \dots, \gamma_n$, and performing an integral Dehn surgery on each of them. 
A careful analysis~\cite{ShannonThese, Shannon} ensures that the obtained flow on the surgered manifold can be reparametrized to be of Anosov type. 
Goodman-Fried surgery changes the underlying manifold by a Dehn surgery, so this procedure changes the orbit equivalence class of the flow. 
One can then define a weaker equivalence relation: two flows~$\Phi_1, \Phi_2$ on two 3-manifolds $M_1, M_2$ are \term{almost equivalent} if they differ by a Goodman-Fried surgery and an orbit equivalence. 

Properties (i) and (iii) in the definition of a Birkhoff section imply that, for $\iota(S)$ a Birkhoff section of a flow~$\Phi$, there is an induced \term{first-return map} $f_S:S\to S$, and so in the open 3-manifold obtained by removing $\iota(\partial S)$ the flow is a suspension flow. 
By filling the boundary tori of this manifold, we see that $\Phi$ is almost equivalent to a suspension flow. 
As noticed by Fried~\cite{FriedAnosov}, when the flow is Anosov, the first-return map is pseudo-Anosov. 
This applies to the geodesic flows we consider. 
It is not hard to see that, in the context of Theorem~\ref{T:G1BS}, the first-return on the closure of $S$ is actually Anosov, {\it i.e.} it has no singularity, so it is conjugated to a linear hyperbolic automorphism of the torus with positive trace. 
This implies that the considered geodesic flow is almost equivalent to te suspension on some mapping torus of the form~$\TT^3_{(\begin{smallmatrix}2&1\\1&1\end{smallmatrix})}$. 
A theorem of Minakawa~\cite{DS} asserts that, under this assumption, the latter flow is almost equivalent to the suspension flow on~$\TT^3_{(\begin{smallmatrix}2&1\\1&1\end{smallmatrix})}$, so one gets 

\begin{introcoro}\label{T:AE2}
If $\OO$ is a 2-dimensional hyperbolic orientable orbifold of small type, then the geodesic flow on~$\U\OO$ is almost equivalent to the suspension flow on~$\TT^3_{(\begin{smallmatrix}2&1\\1&1\end{smallmatrix})}$.
\end{introcoro}


\subsection{Acknowledgements}
I thank Mario Shannon for numerous conversations, and for co-writing a first version of the paper. 
Mario did not want to cosign the second version since the main construction changed and he had no time to check it. 

\subsection{Organisation of the paper}
The construction needed for proving Theorem~\ref{T:G1BS} is rather involved. 
It relies on the description of explicit surfaces in the considered unitary tangent bundles. 
For pedagogic reasons, we first introduce the necessary background in Section~\ref{S:DefGeodesic}. 
Then we introduce a first tool called {\it vertical surfaces} in Section~\ref{S:Vertical}. 
It allows to treat two rather simple cases, namely the case of surfaces (that was already treated in this way by Brunella~\cite{Brunella}) and the case of a sphere with conic points of order~$2$. 
In Section~\ref{S:Horizontal} we introduce another tool called~{\it butterfly surfaces} that allows to treat the case of a sphere with conic points of order at least~$3$ and the case of a genus~$g$ surface with between $2g{+}2$ and~$2g{+}6$ conic points of order at least~$3$.
Finally in Section~\ref{S:General} we combine both tools to treat the general case.

Here is a table summarizing all the cases for the type of the 2-orbifold, and the section in which they are treated: 

\begin{table}[h]
\centering
\begin{tabular}{|c|c|c|c|}
\hline
\multicolumn{2}{|c|}{Sphere} & \multicolumn{2}{c|}{Surface of genus $\ge 1$ with $n$ conic points}  \\ \hline
 $p_1,\dots,p_n=2$ & Section \ref{S:Order2} & $n=0$ & \cite{Brunella}, Section \ref{S:Surface} \\ \hline
 $p_1, \dots, p_n\ge 3$ & Section \ref{S:Sphere}& $2g{+}2\le n\le 2g{+}6$  & Sections \ref{S:Butterfly1}, \ref{S:Butterfly2}\\ \hline
 $p_1, \dots, p_k=2, p_{k+1}, \dots, p_n\ge 3$ & Section \ref{S:GeneralSphere} & $0<n<2g+2 $ & Section \ref{S:GeneralSurface}\\ \hline
\end{tabular}
\end{table}


\section{Preliminaries: geodesic flows on hyperbolic 2-orbifolds}\label{S:DefGeodesic}

Let us explain some basic facts about hyperbolic surfaces, hyperbolic 2-orbifolds, and their geodesic flows. 
We refer the reader to the books \cite{Singer_Thorpe} and \cite{SKatok_FuchsianGroups} for general definitions on Riemannian surfaces, hyperbolic geometry and 2-orbifolds, and to the book~\cite{Katok_Hasselblatt} for the construction of the geodesic flow on the Poincar\'e disk.  

Let $\HH^2$ denote the hyperbolic plane, which is a simply connected Riemannian surface of constant curvature~$-1$. 
For $r>0$ and $a\in\HH^2$, let $\DD(a,r)$ denote the open disk of radius $r$ centered at~$a$. 
For every $p\ge 2$, denote by~$R_p$ the rotation with center~$a$ and angle~$2\pi/p$, and consider the quotient space $\DD_p(a,r):=\DD(a,r)/\langle R_p\rangle$. 
Then $\DD_p(a,r)$ is topologically a disk and the quotient map $\DD(a,r)\to \DD_p(a,r)$ induces a Riemannian structure on $\DD_p(a,r)\setminus\{a\}$, since $R_p$ is an isometry on $\DD(a,r)$. 
Since $p\ge 2$, this structure cannot be extended onto the point $a$ as a Riemannian metric, because the total angle around~$a$ equals $2\pi /p$, which is strictly smaller than $2\pi$. 

\begin{definition}[orientable hyperbolic 2-orbifold]\ 
\begin{enumerate}
\item A \textbf{hyperbolic surface} is a Riemannian surface satisfying that every point has an open neighborhood isometric to $\DD(r)$ in the Poincar\'e disk. Such a Riemannian metric is called a \term{hyperbolic metric} and is characterized by the fact of having constant curvature~$-1$.

\item An \textbf{orientable hyperbolic 2-orbifold} is an orientable surface endowed with a hyperbolic metric in the complement of a set of isolated points, and around each of these points there is a disk of radius $r>0$ which is isometric to some $\DD_p(r)$ with $p\geq 2$. 
Such a point is called a \term{conic point}, and the associated integer~$p$ is called its~\term{order}. 
\end{enumerate}
\end{definition}

Thus, an orientable hyperbolic $2$-orbifold~$\OO$ is a surface~$\Sigma_\OO$ carrying a hyperbolic metric defined everywhere except at some isolated points, where the metric structure is locally obtained from the quotient of a finite group of isometries acting on a hyperbolic disk. 
In the case~$\Sigma_\OO$ is closed, there are finitely many of these conic points. 
Denoting by $g$ the genus of~$\Sigma_\OO$ and by $p_1, \dots, p_n$ the orders of the conic points, the \term{type} of~$\OO$ is then the tuple $(g; p_1, \dots, p_n)$. 

Due to the {Uniformization Theorem} of hyperbolic geometry, every orientable hyperbolic surface is isometric to the quotient $\HH^2/\Gamma$ of hyperbolic plane $\HH^2$ by a free and properly discontinuous action of a subgroup $\Gamma$ of~$\mathrm{Isom}^+(\HH^2)$. 
The uniformization theorem extends to the family of hyperbolic 2-orbifolds\footnote{In Thurston's terminology, this means that all hyperbolic 2-orbifolds are~{\it good}.}.

\begin{theorem}[see e.g. \cite{SKatok_FuchsianGroups}]
Every orientable hyperbolic 2-orbifold $\OO $ is isometric to the quotient~$\HH^2/\Gamma$, where $\Gamma$ is a subgroup of orientation-preserving isometries of $\HH^2$ with properly discontinuous action and finite $\Gamma$-stabilizer on every point. 
\end{theorem}

Given a closed orientable hyperbolic 2-orbifold $\OO$ (including the case of a hyperbolic surface if there is no conic point) there is an associated \term{unitary tangent bundle} and \term{geodesic flow} acting on it. 
The construction is the following: let $\widetilde{\Phi}^t_{\mathrm{geod}}:\U\HH^2\to \U\HH^2$ be the geodesic flow on the unitary tangent bundle of the hyperbolic plane, which is a circle bundle $S^1\hookrightarrow \U\HH^2\to\HH^2$. 
Since $\Gamma$ acts on $\HH^2$ by isometries then it also acts on $\U\HH^2$ through the derivative of its elements, preserving the $S^1$-fibration. 
On $\U\HH^2$ this action is free, i.e. every point has trivial stabilizer. 
To see this, observe that along the $S^1$-fiber of every $z\in\HH^2$ with non-trivial $\Gamma$-stabilizer, the action is a finite group of rotations, so there is no fixed point on these fibers. 
Also since $\Gamma$ acts by isometries on the Poincar\'e disk then the geodesic flow is $\Gamma$-equivariant. One thus sets:

\begin{definition}[geodesic flow]\ For $\OO $ an orientable hyperbolic 2-orbifold identified with~$\HH^2/\Gamma$, 
\begin{itemize}
\item its \textbf{unitary tangent bundle} is the 3-manifold $\U\OO :=\U\HH^2/\Gamma$; 
\item it carries a flow $\fgeod^t:\U\OO \to \U\OO $ induced from $\widetilde{\Phi}^t_{\mathrm{geod}}$ and called the \textbf{geodesic flow}.
\end{itemize}
\end{definition}

It is worth noting that since the action of $\Gamma$ on $\U\HH^2$ through the derivative of its elements preserves the $S^1$-fibers then the quotient $\U\OO $ is endowed with a foliation by circles. 
Nevertheless, this foliation needs not to be a circle bundle as in the case of a hyperbolic surface. 
It fails to be locally trivial around the fibers corresponding to the conic points. 
This type of fibration is called \emph{Seifert fibration}. See the notes~\cite{Hatcher_Seifert} for a survey on Seifert manifolds.\footnote{Note that, if $\OO $ is a 2-orbifold with non-empty set of conic points, then its unitary tangent bundle as a 2-orbifold is not homeomorphic to the unitary tangent bundle of the underlying topological surface $\Sigma_\OO$. In particular the foliation by fibers on $\U\OO$ is a Seifert fibration and not a fibration by circles.} 

Note that the orbits of the geodesic flow~$\fgeod$ correspond to lifts of oriented geodesics. 

\begin{notation}[lifts of geodesic arcs]
On an orientable 2-orbifold~$\OO$, for $\alpha$ an oriented arc of geodesic, we denote by~$\vec\alpha$ its oriented lift in~$\U\OO$. 
We denote by~$\revec\alpha$ the lift of~$\alpha$ with the opposite orientation, and by~$\vrevec\alpha$ the union~$\vec\alpha\cup\revec\alpha$. 
\end{notation}

Remark the following dichotomy: if $p$ is a conic point of odd order, then the geodesic arcs go through $p$; on the other hand if $p$ is a conic point of even order, then the geodesic arcs make a U-turn at~$p$. 
In that case, one has~$\vec\alpha=\revec\alpha=\vrevec\alpha$. 

\medskip

Gromov remarked that given two hyperbolic surfaces of the same genus, the associated geodesic flows on the unitary tangent bundles are equivalent. 
The statement extends to 2-dimensional orbifolds, with the same proof:

\begin{theorem}\cite{Gromov}\label{Gro}
Given two compact hyperbolic orientable 2-orbifolds $\OO_1, \OO_2$ of the same type, then there exists a homeomorphism $\U\OO_1\to\U\OO_2$ that sends the oriented orbits of the geodesic flow on $\U\OO_1$ onto the oriented orbits of the geodesic flow on $\U\OO_2$.
\end{theorem}

\begin{proof}
Since $\OO_1$ and $\OO_2$ are hyperbolic, their universal cover is~$\Hy$ and they are isometric to $\Hy/\Gamma_1$ and~$\Hy/\Gamma_2$ respectively. 
Since they are of the same type, there is an isomorphism $f : \Gamma_1 \to \Gamma_2$. 
Identifying $\partial \Hy$ with~$\partial \Gamma_1$ and $\partial\Gamma_2$, $f$ extends to a $(\Gamma_1, \Gamma_2)$-equivariant homeomorphism $\partial \Hy\to\partial\Hy$.

Now a geodesic on $\Hy$ is represented by a pair of distinct points on $\partial \Hy$ and a unitary tangent vector by a positively oriented triple of distinct points (the third point defines a unique canonical projection on the geodesics represented by the first two points). 
Denoting by $C_3(\partial\Hy)$ the set of triples of distinct points in~$\partial\Hy$, $f$ extends to a $(\Gamma_1, \Gamma_2)$-equivariant homeomorphism $C_3(\partial\Hy)\to C_3(\partial\Hy)$, that is, a homeomorphism~$\U\Hy\to\U\Hy$. 
Projecting on the first two coordinates one sees that it sends oriented geodesics onto oriented geodesics. 
Notice that since the third coordinate is not the time-parameter, the speed is not at all preserved. 
Projecting back to $\U\Hy/\Gamma_1=\U\OO_1$ and $\U\Hy/\Gamma_2=\U\OO_2$ , we obtain the desired orbital equivalence.
\end{proof}


\section{Vertical surfaces and two simple cases}\label{S:Vertical}

The proof of Theorem~\ref{T:G1BS} relies on the description of explicit surfaces in the considered unitary tangent bundles. 
The construction being rather involved, we first introduce one tool to construct and describe specific surfaces, that allows to prove two particular cases of the theorem, namely the case of a surface (that was already treated using this tool by Brunella~\cite{Brunella}) and the case of a sphere with at least $5$ conic points of order~2. 

\subsection{Vertical surfaces}\label{S:VerticalDef}

An arc~$\alpha$ in a 2-orbifold is said~\term{simple} if it is embedded and its interior does not meet any conic point. 

\begin{definition}(after~\cite{Brunella, CD})\label{D:Vertical}
Given a hyperbolic 2-orbifold~$\OO$, a geodesic simple arc~$\alpha$ in~$\OO$, and a choice~$s$ of a side of~$\alpha$ (or, equivalently, of a coorientation of~$\alpha$), the set of those unitary tangent vectors based at points of~$\alpha$ and pointing toward the side~$s$ is called a \term{vertical rectangle} in~$\U\OO$. It is denoted by~$R(\alpha, s)$, see Figure~\ref{F:Rec} left.
\end{definition}

The surface~$R(\alpha, s)$ of the above definition has the topology of a disc. 
Its interior is transverse to~$\fgeod$. 
Its boundary is made of four arcs : the two arcs~$\vrevec\alpha=\vec\alpha\cup\revec\alpha$ that we call \term{horizontal}, and two arcs that correspond to the parts of the fibers of the extremities of~$\alpha$ pointing toward the side~$s$ that we call~\term{vertical} (see Figure~\ref{F:Rec} left).

The coorientation of the interior of~$R(\alpha, s)$ by~$\fgeod$ induces an orientation of~$R(\alpha, s)$, which in turns induces an orientation of its boundary. 
One checks that this orientation of the horizontal boundary~$\vrevec\alpha$ is {\it opposite} to the orientation of these two arcs given by~$\fgeod$ (see Figure~\ref{F:Rec} left). 
Also, when looking toward~$s$, the vertical boundary in the fiber of the left extremity of~$\alpha$ is oriented upward while the boundary in the fiber of the right extremity goes downward.  

\medskip

Now assume that two geodesic simple arcs meet at a unique point~$p$ on an hyperbolic 2-orbifold~$\OO$, thus delimiting four arcs~$\alpha_1, \dots, \alpha_4$ with a common extremity~$p$ (see Figure~\ref{F:Rec} center).
These arcs decompose a neighborhood of~$p$ into four local sectors that can be checkerboard-colored. 
For each arc~$\alpha_i$, choose the side~$s_i$ that corresponds to the white face. 
Then the four rectangles~$R(\alpha_1, s_1)\cup\dots\cup R(\alpha_4, s_4)$ glue into a topological surface~$S_p$ with boundary in~$\U\OO$, which also inherits a coorientation from~$\fgeod$, and hence an orientation. 
The boundary of~$S_p$ consists of the eight segments of orbits $\vrevec\alpha_1\cup\dots\cup\vrevec\alpha_4$ of~$\fgeod$, plus four vertical arcs in the fibers of the extremities of~$\alpha_1, \dots, \alpha_4$ that are not~$p$. 
Indeed, in the fiber~$\U p$ and taking into account the orientations, the vertical parts of the four rectangles two-by-two cancel. 
The surface is only topological since it is not smooth along the fiber of~$p$. 
However one can perform an arbitrarly small isotopy that fixes the horizontal boundary and makes the interior smooth and transverse to~$\fgeod$ (as in Figure~\ref{F:Rec} right). 
This smoothing procedure near the fiber of the double-point~$p$ will be used repeatedly in the sequel. 

\begin{figure}[ht]
\begin{picture}(130,65)(0,0)
\put(-5,5){\includegraphics[width=.4\textwidth]{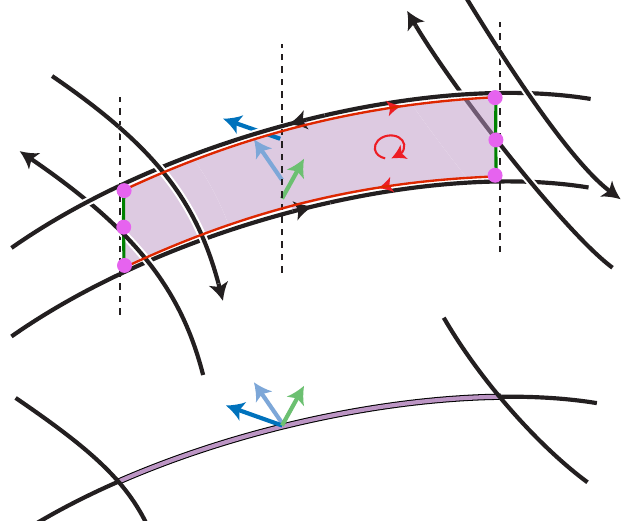}}
\put(65,0){\includegraphics[width=.42\textwidth]{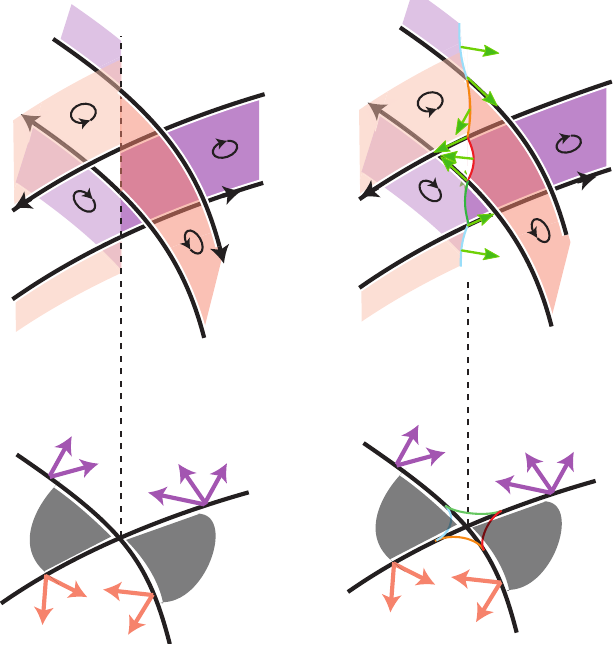}}
\put(23,47){$R(\alpha, s)$}
\put(20,19){$s$}
\put(28,13){$\alpha$}
\end{picture}
\caption{On the left, a geodesic simple arc~$\alpha$ on a 2-orbifold~$\OO$, a choice of a side~$s$, and the corresponding rectangle~$R(\alpha, s)$ in~$\U\OO$. 
In the center, the intersection of two geodesics on~$\OO$ forms a graph with four arcs adjacent to one vertex. 
Considering a (local) checkerboard coloring and the sides of the arcs determined by the white face, we obtain a (local) topological surface in~$\U\OO$ made of four rectangles. 
On the right, a smoothing of that surface that turns it into a smooth surface transverse to the geodesic flow. }
\label{F:Rec}
\end{figure}
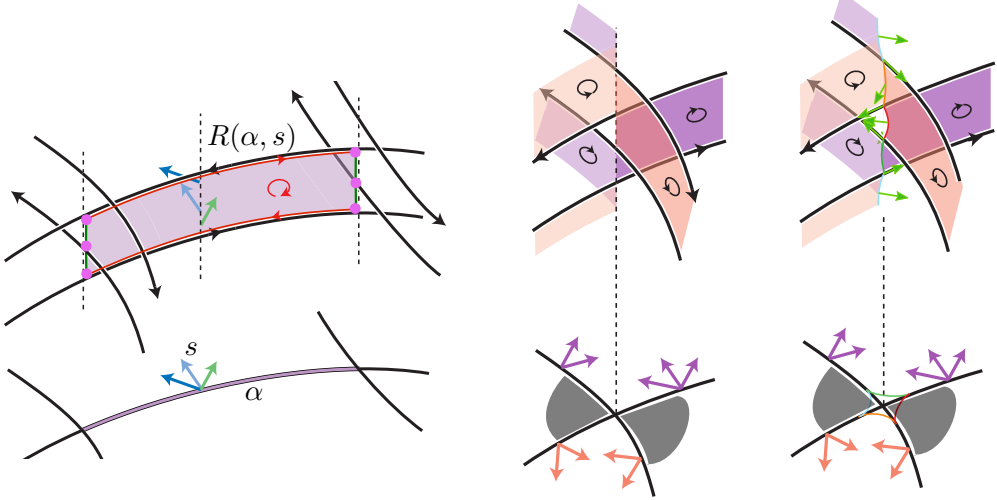


\subsection{Proof of Theorem~\ref{T:G1BS} in the case of a surface~\cite[Section 1, Description~2]{Brunella}}\label{S:Surface}

\subsubsection{Choice of a suitable orbifold metric}\label{S:MetricSurface}
We assume~$g\ge 2$ for otherwise there is no hyperbolic surface of genus~$g$. 
Thanks to Gromov's theorem~\ref{Gro}, the metric is not relevant concerning the existence and the topology of Birkhoff sections for the geodesic flow, as long as it is hyperbolic. 
However, choosing a suitable hyperbolic metric will help in describing and picturing the construction.\footnote{In this first peculiar case, the choice of the metric is actually not relevant, but it will be useful later for proving more general cases.} 

The following construction is displayed in Figure~\ref{F:CasSurface}. 
Consider a regular right-angled $2g{+}2$-gon~$A$ in~$\Hy$, and denote its vertices by~$a_1, \dots, a_{2g+2}$. 
Consider the symmetry~$s_h$ with axis~$(a_1a_2)$, and set $B:=s_h(A)$.  
Also consider the symmetry~$s_v$ with axis~$(a_1a_{2g+2})$, and set $D:=s_v(A)$. 
Finally consider the composition $r:=s_{h}\circ s_{v}$ which is a $\pi$-rotation around~$a_1$, and set $C:=r(A)$

The polygon~$A\cup B\cup C\cup D$ is a $8g{-}4$-gon. 
Consider the identification of its sides given as follows (see left of Figure~\ref{F:CasSurface}): 
for every $i$ in~$\{0, \dots, g\}$ and counting mod $2g{+}2$, identify 
\begin{itemize}
\item $[a_{2i}a_{2i+1}]$ with $s_{v}([a_{2i}a_{2i+1}])$, 
\item $r([a_{2i}a_{2i+1}])$ with $s_{h}([a_{2i}a_{2i+1}])$, 
\item $[a_{2i+1}a_{2i+2}]$ with $s_{h}([a_{2i+1}a_{2i+2}])$, 
\item $r([a_{2i+1}a_{2i+2}])$ with $s_{v}([a_{2i_1}a_{2i+2}])$. 
\end{itemize}
In other words, edges of $A$ and $C$  are identified with their images under $s_{v}$ and $s_h$ in $B$ and $D$ alternatively. 

In the quotient, every vertex~$a_i$ is identified with $s_v(a_i), s_h(a_i)$ and $r(a_i)$, so that the quotient is a hyperbolic surface that we denote by~$\Sigma_g$. 
We keep the notation~$a_i$ for the projection of~$a_i$ in the quotient. 
Also we denote by $e_i$ the projection of the edge~$[a_ia_{i+1}]$ in the quotient and by $e'_i$ the projection of the edge~$[r(a_i)r(a_{i+1})]$. 
In this way the edges of~$A$ are $e_1, e_2, \dots, e_{2g+2}$ in trigonometric order, the edges of~$B$ are $e_1, e'_2, e_3, \dots, e'_{2g+2}$ in clockwise order, the edges of~$C$ are $e'_1, e'_2, \dots, e'_{2g+2}$ in trigonometric order, and the edges of~$D$ are~$e'_1, e_2, e'_3, \dots, e_{2g+2}$ in clockwise order. 

The faces, sides and vertices of~$A\cup B\cup C\cup D$ then induce a graph~$G_g$ on~$\Sigma_g$ with~$4$ faces (that correspond to $A$, $B$, $C$, and $D$), $4g{+}4$ edges (namely $e_1, e'_1, \dots, e_{2g+2}, e'_{2g+2}$), and $2g{+}2$ vertices, so that $\Sigma_g$ has genus~$g$. 
Picturing~$\Sigma_g$ as on the right of Figure~\ref{F:CasSurface}, the two faces~$A$ and $B$ lie on the top, and~$C$ and $D$ lie on the bottom. 
Also~$B$ and $C$ lie on the front, and~$A$ and $D$ lie on the back. 

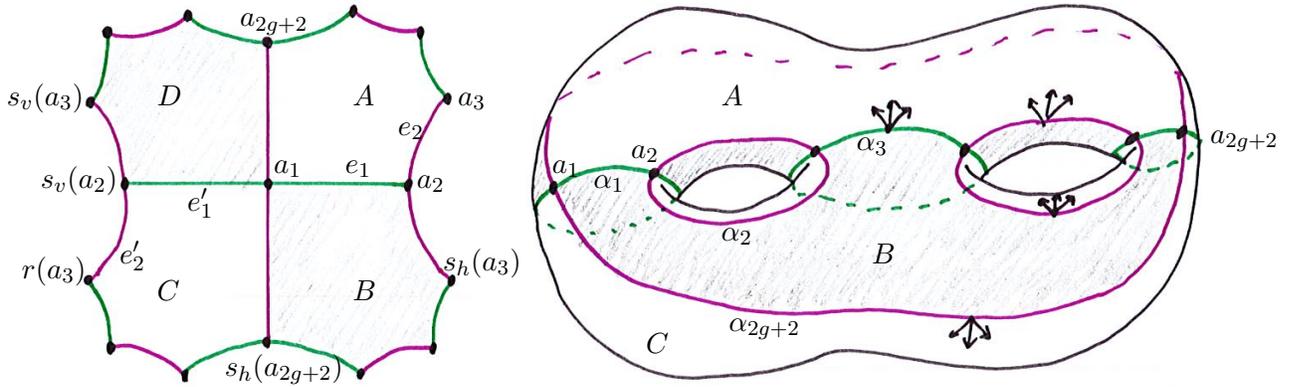
\begin{figure}[ht]
\begin{picture}(150,55)(0,0)
\put(0,0){\includegraphics[width=.36\textwidth]{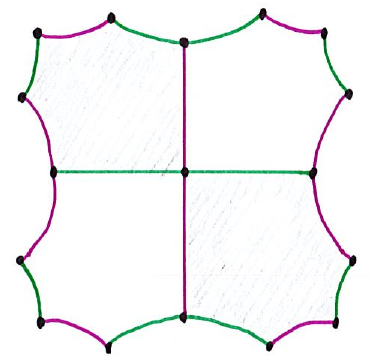}}
\put(60,0){\includegraphics[width=.62\textwidth]{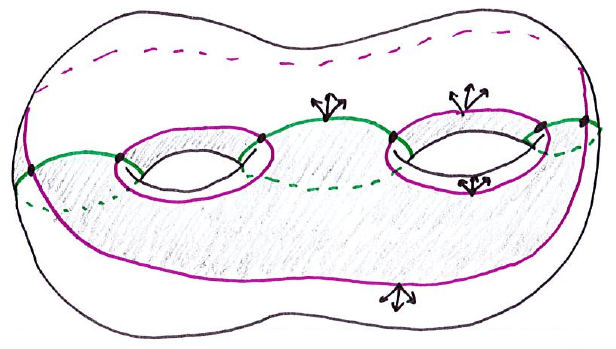}}
\put(38,38){$A$}
\put(38,12){$B$}
\put(12,12){$C$}
\put(12,38){$D$}
\put(27.5,29){$a_1$}
\put(46.5,27){$a_2$}
\put(-3.5,27){$s_v(a_2)$}
\put(23,48.5){$a_{2g+2}$}
\put(21,2){$s_h(a_{2g+2})$}
\put(52,38){$a_3$}
\put(50,16){$s_h(a_3)$}
\put(-6,15){$r(a_3)$}
\put(-8,38){$s_v(a_3)$}
\put(37,29){$e_1$}
\put(16,24){$e'_1$}
\put(44,34){$e_2$}
\put(7,17){$e'_2$}
\put(87,38){$A$}
\put(107,17){$B$}
\put(77,5){$C$}
\put(64.5,29){$a_1$}
\put(75,31){$a_2$}
\put(152,33){$a_{2g+2}$}
\put(70,27){$\alpha_1$}
\put(87,20){$\alpha_2$}
\put(105,32){$\alpha_3$}
\put(88,8){$\alpha_{2g+2}$}
\end{picture}
\caption{On the left, the $2g{+}2$-gon~$A$, and its images~$B, D, C$ under the symmetries~$s_h$ and~$s_v$ and the $\pi$-rotation~$r$. 
On the right, the hyperbolic surface~$\Sigma_g$ obtained after identifications of the sides of $A, B, C, D$, with the $2g{+}2$ closed geodesics~$\alpha_1, \dots, \alpha_{2g+2}$ on it. 
These curves yield a graph~$G_g$ on~$\Sigma_g$. 
}
\label{F:CasSurface}
\end{figure}

\subsubsection{Choice of the boundary orbits}\label{S:BoundarySurface}
On the surface~$\Sigma_g$, for every $i$ in $\{1, \dots, 2g{+}2\}$, note that $e_i$ and $e'_i$ meet in $a_i$ and $a_{i+1}$ with an angle~$\pi$, so that their concatenation is a closed geodesic on~$\Sigma_g$, that we denote by~$\alpha_i$. 
Consider the link~$L_g:=\vrevec\alpha_1\cup\dots\cup\vrevec\alpha_{2g{+}2}$. 

\subsubsection{Choice of the surface}\label{S:SurfaceSurface}
Color the faces~$A$ and $C$ in white, and $B$ and $D$ in black. 
This coloring is checkerboard-like around every edge and vertex of~$G_g$. 
For each edge of type~$e_i$ or~$e'_i$, choose the side given by the white face and denote it by~$s_i$ or $s'_i$ respectively. 
Now consider the surface~$S^\times_g:=R(e_1, s_1)\cup R(e'_1, s'_1)\cup\dots\cup R(e'_{2g+2}, s'_{2g+2})$ made of~$4g{+}4$ vertical rectangles, each cooriented, hence oriented, by the geodesic flow. 
The horizontal boundary of~$S^\times_g$ then consists of the link~$L_g$ (with the orientation that is opposite to the geodesic flow, as explained in Section~\ref{S:VerticalDef} and on Figure~\ref{F:Rec} left). 
Indeed, all the vertical components of the boundaries of the rectangles cancel two-by-two. 

Around every vertex of~$a_1, \dots, a_{2g+2}$, there are four adjacent arcs. 
So around the fibers~$\U a_1, \dots, \U a_{2g+2}$, the surface~$S^\times_g$ is made of four rectangles. 
As explained at the end of the previous section and on the right of Figure~\ref{F:Rec}, one can smooth~$S^\times_g$ by a small isotopy supported in a neighborhood of the fibers of the vertices of~$G_g$ into a surface with the same boundary~$L_g$ and whose interior is still transverse to~$\fgeod$. 
Denote by~$S_g$ the resulting surface. 

\subsubsection{Computation of the genus}\label{S:GenusSurface}
The surfaces~$S^\times_g$ and~$S_g$ being isotopic, they have the same topology. 
We therefore only compute the Euler characteristic of~$S^\times_g$. 
The latter is made of the~$4g{+}4$ rectangles $R(e_1, s_1), R(e'_1, s'_1), \dots, R(e'_{2g+2}, s'_{2g+2})$ which are glued along their vertical boundaries. 
One must pay attention that for describing~$S^\times_g$ properly, one has to subdivide every vertical arc into two subarcs that are glued to two different adjacent rectangles, as in the center of Figure~\ref{F:Rec}. 

After this decomposition, every vertical rectangle of type~$R(e_i, s_i)$ or $R(e'_i, s'_i)$ is a hexagon with $2$ horizontal sides and $4$ vertical sides. 
Every horizontal side remains on the boundary, so it contributes by~$-1$ to the Euler characteristic. 
Every vertical side is identified with one other vertical side (in green on Figure~\ref{F:Rec}), so it only contributes by~$-\frac12$ to the Euler characteristic. 
Finally every vertex (in pink on Figure~\ref{F:Rec}) is identified with two other vertices, so it contributes by~$+\frac13$ to the Euler characteristic. 
All-in-all, the total contribution of~$R(e_i, s_i)$ to~$\chi(S^\times_g)$ is $+1-2\cdot 1-4\cdot \frac12+6\cdot\frac13=-1$. 

When summing over all rectangles, we find that $\chi(S^\times_g)$ equals~$-4g{-}4$. 
Since $S^\times_g$ has $4g{+}4$ boundary components, its genus is~$1$.  

\medskip

\begin{remark}\label{R:Genus}
As finding Birkhoff section with genus exactly one is the hard part, we provide an alternative way of computing the genera of our constructions. 
It relies on Poincaré-Hopf formula and the fact that the surfaces we construct are transverse in their interior to the geodesic flows~$\fgeod$. 
The latter being of Anosov type, it admits invariant foliations~$\Fs, \Fu$ (actually only one is needed for us). 
The intersection~$\Fs\cap S_g$ is then a foliation on~$S_g$, which is regular in the interior of~$S_g$. 
However it has singularties on~$\partial S_g$: these arise when the tangent plane~$\mathrm{T}S_g$ coincide with the tangent plane~$\mathrm{T}\Fs$. 
In that case, the foliation~$\Fs\cap S_g$ has a singularity of index~$-1/2$. 
Notice that, if our surfaces are Birkhoff sections, every boundary component needs to have a positive number of singularities (as $\mathrm{T}S_g$ needs to intersect~$\Fs$ for the first-return time to be bounded), and this number is actually even because of orientability. 
Now, by the Poincaré-Hopf formula, the Euler characteristic of~$S_g$ is the sum of all the indices of the singularities of~$\Fs\cap S_g$. 
Every boundary component has a negative integral contribution. 
Therefore~$S_g$ has genus one if and only if every boundary components contributes by~$-1$. 

In order to compute the contribution of a given boundary orbit, it may be easier to compute how~$\mathrm{T}S_g$ intersects the fiber-direction (which yields the same framing as~$\Fs$). 
In the present situation, one checks that $S_g$ is tangent to the vertical around~$\vec\alpha_i$ exactly once at each intersection points of~$\alpha_i$, that is in~$\U a_i$ and $\U a_{i+1}$. 
Therefore the contribution of ~$\vec\alpha_i$ to~$\chi(S_g)$ is exactly $-1$. 
Hence $S_g$ is indeed a torus. 
\end{remark}

\subsubsection{Intersection with the orbits of the geodesic flow}\label{S:IntersectionSurface}
One must check that every orbit of the geodesic flow~$\fgeod$ intersects the surface~$S_g^\times$ in bounded time. 

Consider the following oriented graph~$G^*_g$ which is an oriented dual of~$G_g$, see Figure~\ref{F:GrapheSurface}: 
it has four vertices labeled~$A^*, B^*, C^*, D^*$;
for every edge~$e_i$ of~$G_g$, there are two edges $e^+_i$ and $e^-_i$ respectively corresponding to crossing $e_i$ from the black face to the white face or from the white face to the black face, and for every vertex~$a_i$ of~$G_g$, there are four edges $a_i^{A\to C}, a_i^{C\to A}, a_i^{B\to D}, a_i^{D\to B}$ corresponding to crossing~$a_i$ diagonally from one adjacent face to the opposite one. 
Some edges are dotted and some are plain, as shown on Figure~\ref{F:GrapheSurface}.

\begin{figure}[ht]
\begin{picture}(70,63)(0,0)
\put(0,0){\includegraphics[width=.5\textwidth]{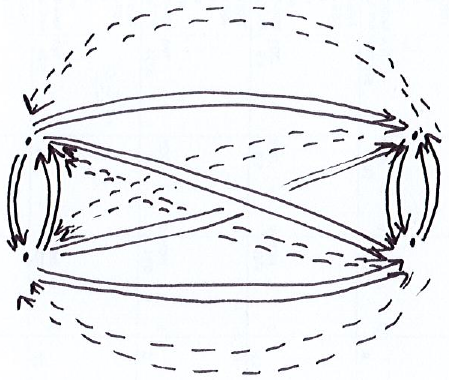}}
\put(74,42){$A^*$}
\put(70,19){$C^*$}
\put(-1,41){$B^*$}
\put(-2,17){$D^*$}
\put(35,50){$e_{2i+1}^+$}
\put(55,60){$e_{2i+1}^-$}
\put(55,32){$e_{2i}^+$}
\put(36,41){$e_{2i}^-$}
\put(35,10){$e_{2i+1}'^+$}
\put(12,3){$e_{2i+1}'^-$}
\put(18,40){$e_{2i}'^+$}
\put(37,20){$e_{2i}'^-$}
\put(75,29){$a_{i}^{C\to A}$}
\put(-10,29){$a_{i}^{B\to D}$}
\end{picture}
\caption{The dual graph~$G^*_g$. 
Every depicted edge has in fact many parallel copies with the same origin and destination. 
Given a geodesic~$\gamma$ on~$\Sigma_g$, the faces and edges visited by~$\gamma$ describe a path in~$G^*_g$. 
Every time this path uses a bold edge corresponds to an intersection point of the orbit~$\vec\alpha$ with the surface~$S_g$ in~$\U\Sigma_g$. 
}
\label{F:GrapheSurface}
\end{figure}
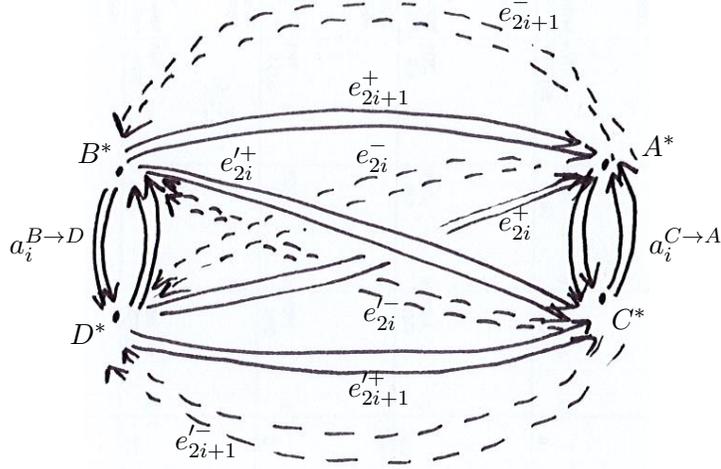

Consider an arbitrary oriented arc of geodesic~$\gamma$ on~$\Sigma_g$, and consider its lift~$\vec\gamma$ in~$\U\Sigma_g$ which is an arc of orbit of~$\fgeod$ (and every arc of orbit is of this form). 
To~$\gamma$ one associates a path~$\gamma^*$ in~$G^*_g$ obtained by following the list of faces visited by~$\gamma$ and the list of oriented edges (and vertices) crossed by~$\gamma$. 

Now, every time $\gamma$ goes from $B$ or $D$ to $A$ or $C$, its lift~$\vec\gamma$ intersects~$S_g$ positively. 
This is also the case every time~$\gamma$ goes through one vertex~$a_i$.
Therefore, every time the path~$\gamma^*$ goes along one bolded edge of~$G^*_g$ corresponds to one intersection  of~$\gamma$ with~$S_g$. 
Let $d_g$ denote the diameter of the polygon~$A$ in~$\Hy$. 
The key-point is that at least every second edge used by~$\gamma^*$ is bolded. 
Therefore, if $\gamma$ has length at least~$2d_g$, its lift~$\vec\gamma$ intersects~$S_g$ positively. 

This concludes the proof of Theorem~\ref{T:G1BS} in this case. 


\subsection{Proof of Theorem~\ref{T:G1BS} when $g=0$, $p_1=\dots=p_n=2$}\label{S:Order2}
Fix a positive integer~$n\ge 5$ as smaller values of~$n$ correspond to non-hyperbolic orbifolds.

\subsubsection{Choice of a suitable orbifold metric}\label{S:MetricOrder2}
As in the previous case, thanks to Gromov's theorem~\ref{Gro}, the metric is not relevant concerning the existence and the topology of Birkhoff sections for the geodesic flow, as long as it is hyperbolic. 
However, choosing a suitable hyperbolic metric will help in describing and picturing the construction.

Consider a right-angled regular $n$-gon $E$ in~$\Hy$, and denote by~$b_1, \dots, b_n$ its vertices. 
Consider the image~$F$ of $E$ under a reflection across~$(b_1b_2)$, and denote its vertices by $b_1, b_2, b'_3, \dots, b'_n$. 
Identify every side $[b_ib_{i+1}]$ of ~$E$ with the side  $[b'_ib'_{i+1}]$ of~$F$. 
The quotient of $E\cup F$ under these identifications is a hyperbolic orientable 2-orbifold that is a sphere with $n$ conic points of order 2, coming for the vertices~$b_1, \dots, b_n$. 
Denote it by~$\OO_{0; 2, \dots, 2}$. 
The polygons $E$ and $F$ induce two hemispheres on~$\OO_{0; 2, \dots, 2}$, while the union of the segments~$[b_1b_2], [b_2b_3],\dots, [b_nb_1]$ forms an equator. 

\begin{figure}[ht]
\begin{picture}(110,60)(0,0)
\put(0,0){\includegraphics[width=.27\textwidth]{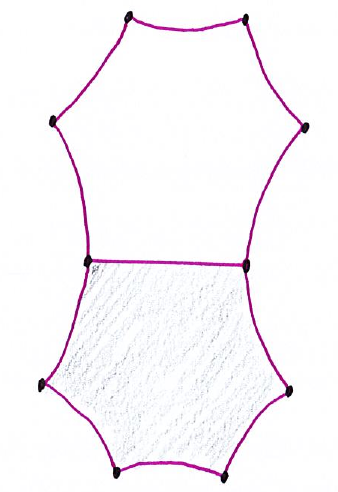}}
\put(50,5){\includegraphics[width=.42\textwidth]{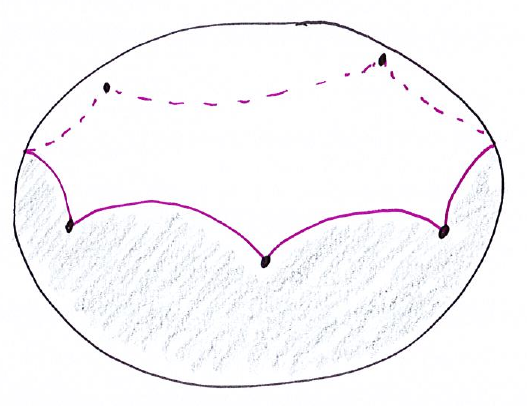}}
\put(19,38){$E$}
\put(18,14){$F$}
\put(6.5,26){$b_1$}
\put(30.5,26){$b_2$}
\put(38,42){$b_3$}
\put(2,43.3){$b_n$}
\put(36,11){$b'_3$}
\put(0,11){$b'_n$}
\put(81,36){$E$}
\put(72,14){$F$}
\put(82,19){$b_1$}
\put(103.5,23){$b_2$}
\put(58.5,23){$b_n$}
\put(91,30){$\beta_1$}
\put(102,34){$\beta_2$}
\put(70,31){$\beta_n$}
\end{picture}
\caption{The orbifold~$\OO_{0; 2, \dots, 2}$, with the $n$ conic points~$b_1, \dots, b_n$ of order~$2$ on the equator. 
}
\label{F:CasSphere}
\end{figure}
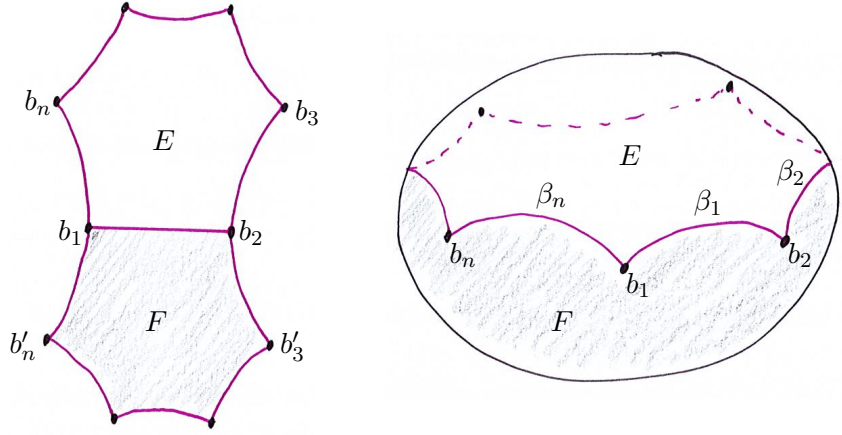

\subsubsection{Choice of the boundary orbits}\label{S:BoundaryOrder2}
For every~$i$, the edge~$[b_ib_{i+1}]$ (counting mod~$n$) connects two conic points of order~$2$. 
Therefore it forms a periodic geodesic on~$\OO_{0; 2, \dots, 2}$ that we denote by~$\beta_i$.
Its lift~$\vrevec \beta_i$ is a single periodic orbit of~$\fgeod$. 
Consider the $n$-component link~$L_{0;2, \dots, 2}:=\vrevec\beta_1\cup\dots\cup\vrevec\beta_n$. 

\subsubsection{Choice of the surface}\label{S:SurfaceOrder2}
Color $E$ in white and $F$ in black. 
Also, for every~$i$, choose $s_i$ to be the white side of~$\beta_i$. 
Consider the surface~$S_{0; 2, \dots, 2}^\times := R(\beta_1, s_1)\cup\dots\cup R(\beta_n, s_n)$. 

For every $i$, in a local degree-2 orbifold chart around~$b_i$, the arcs~$[b_{i-1}b_i]$ and~$[b_{i}b_{i+1}]$ lift to two intersecting geodesics, and the coloring lifts to a checkerboard coloring as at the bottom left of Figure~\ref{F:Order2}. 
Therefore the surface~$S_{0; 2, \dots, 2}^\times$ lifts to a surface similar to that at the top left of Figure~\ref{F:Order2}. 
One can then perform a smoothing isotopy that is invariant by the order-2 rotation around~$b_i$ and obtain a smooth surface~$S_{0; 2, \dots, 2}$ that is the quotient of the top left of Figure~\ref{F:Order2} by the order-2 rotation, see the result of the top right of Figure~\ref{F:Order2}.

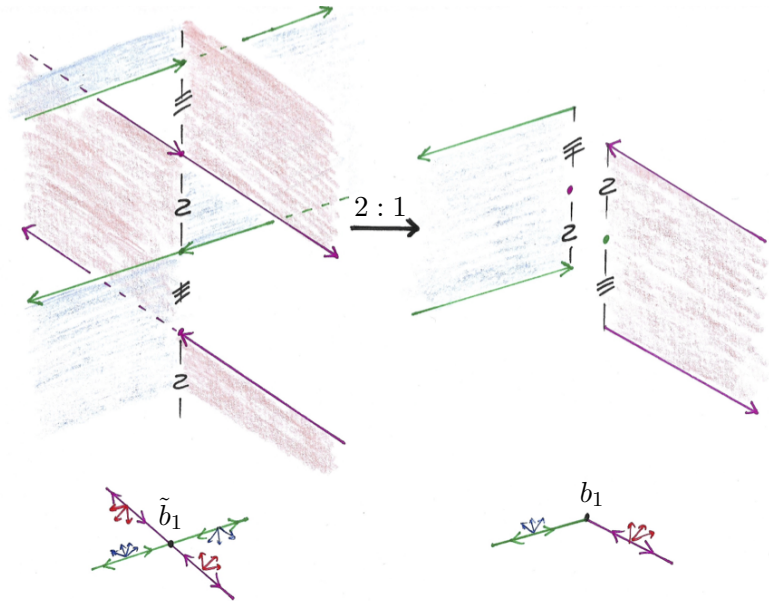
\begin{figure}[ht]
\begin{picture}(100,80)(0,0)
\put(0,0){\includegraphics[width=.7\textwidth]{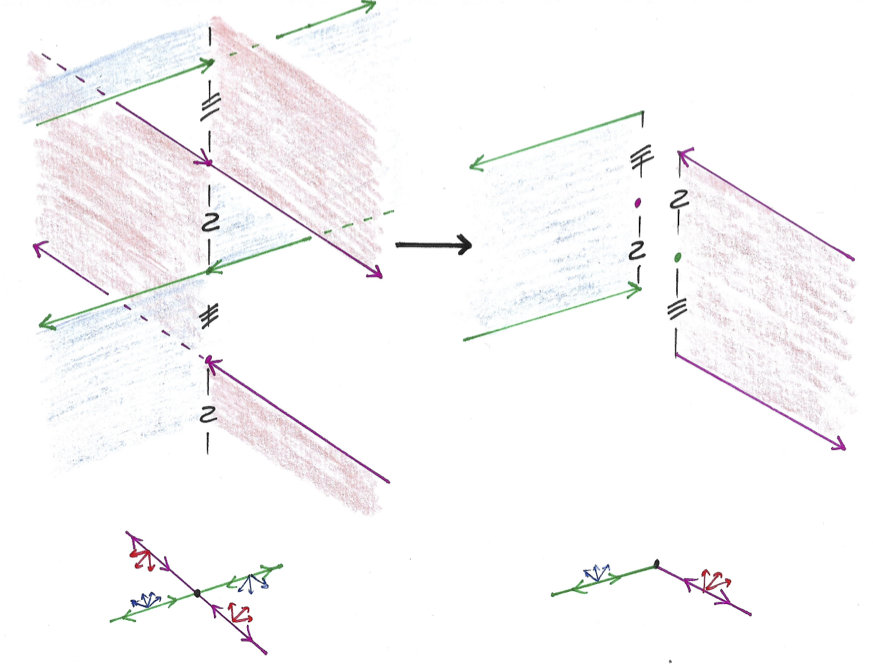}}
\put(78,14){$b_1$}
\put(22,11){$\tilde b_1$}
\put(48,52){$2:1$}
\end{picture}
\caption{In the case $p_{i-1}=p_i=p_{i+1}=2$, the surfaces~$R(\beta_{i-1}, s_{i-1})$ and $R(\beta_{i}, s_{i})$ around the fiber~$\U b_i$ (right) and in the degree-2 cover (left).}
\label{F:Order2}
\end{figure}

\subsubsection{Computation of the genus}\label{S:GenusOrder2}

The surfaces~$S_{0; 2, \dots, 2}^\times$ and $S_{0; 2, \dots, 2}$ being isotopic, we only compute the Euler characteristic of~$S_{0; 2, \dots, 2}^\times$. 
It is made of~$n$ rectangles, glued along their vertical boundaries. 
As in the previous case, one can subdivide the boundary of these rectangles so as to obtain hexagons whose sides and vertices are glued together. 
For every~$i$, the same computation as in Section~\ref{S:GenusSurface} shows that total contribution of~$R({\beta_i, s_i})$ is~$+1-2\cdot 1-4\cdot\frac12+6\cdot\frac13=-1$. 
Therefore one has~$\chi(S_{0; 2, \dots, 2}^\times)=-n$. 
Since it has $n$ boundary components, $S_{0; 2, \dots, 2}^\times)$ has genus~$1$, and so does $S_{0; 2, \dots, 2}$. 

\subsubsection{Intersection with orbits of the geodesic flow}\label{S:IntersectionOrder2}
Consider an oriented arc of geodesic~$\gamma$ on~$\OO_{0; 2, \dots, 2}$ and its lift~$\vec\gamma$ in~$\U\OO_{0; 2, \dots, 2}$. 
Every time $\gamma$ goes from $F$ to~$E$, $\vec\gamma$ intersects~$S_{0; 2, \dots, 2}$. 
Note that an oriented geodesic can go directly from $E$ to $E$ (or from $F$ to $F$) if it goes through a vertex~$b_i$. 
In this case one checks that its lift crosses~$S_{0; 2, \dots, 2}$, as does any small $C^1$-perturbation. 
Therefore, denoting by~$d_n$ the diameter of~$E$, we see that, as soon as the length of~$\gamma$ is more than~$2d_n$, its lift~$\vec\gamma$ intersects~$S_{0; 2, \dots, 2}$ positively. 
This concludes the proof that~$S_{0; 2, \dots, 2}$ is a genus-one Birkhoff section for~$\fgeod$ in this case. 


\section{Butterfly surfaces, the case of spheres with conic points of large order, and the case of surfaces of positive genus with $2g+2$ conic points of large order}\label{S:Horizontal}

In order to treat new cases compared to the previous section, we need other types of surfaces that are not any more tangent to the fibration of~$\U\OO$ by fibers but transverse to it. 
Such surfaces are called horizontal. 
Here we need horizontal surfaces of a particular type that we call \emph{butterfly surfaces}. 

\subsection{Butterfly surfaces}\label{S:Butterfly} For $P$ a polygon in a surface, we denote by~$P^\circ$ the complement of the vertices of~$P$, that is, the union of the interior of~$P$ and the interior of the sides of~$P$. 

\begin{definition}(see Figure~\ref{F:Horizontal} left) 
Given a hyperbolic 2-orbifold~$\OO$, a polygon~$P$ in~$\OO$ with geodesic boundary and containing no conic point, and a unitary vector field~$v$ on~$P^\circ$ that is tangent to the sides of~$P$, the set~$H(P,v):=\overline{\{(p, v(p)\,;\, p\in P^\circ\}}\subset\U P$ is called the~\term{horizontal surface} in~$\U\OO$ associated to~$P$ and $v$. 
\end{definition}

The surface~$H(P,v)$ defined above is (the closure of) a section of the unitary tangent bundle of~$P^\circ$, hence it has the topology of a disc. 
Denoting by~$e_1, \dots, e_n$ the sides of~$P$, its boundary consists of~$2n$ \term{horizontal} segments, namely the $n$~segments denoted by $s(e_1,v), \dots, s(e_n, v)$ that are above~$e_1, \dots, e_n$ respectively, and $n$ \term{vertical} segments in the fibers of the $n$ vertices of~$P$ that correspond to the changes of direction of $v$ around that vertex. 
Note that since $v$ is tangent to every side~$e_i$, then the lift $s(e_i,v)$ coincides with either	~$\vec e_i$ or~$\revec e_i$. 

The surface~$H(P,v)$ is not transverse to~$\fgeod$ when the integral curves of~$v$ have inflection points. 
However, if the integral curves of~$v$ have non-vanishing curvature, then~ the interior of $H(P,v)$ is transverse to~$\fgeod$, and can then be cooriented by the flow. 
This condition will always be satisfied in our constructions. 

\begin{figure}[ht]
\begin{picture}(130,115)(0,0)
\put(0,0){\includegraphics[width=.35\textwidth]{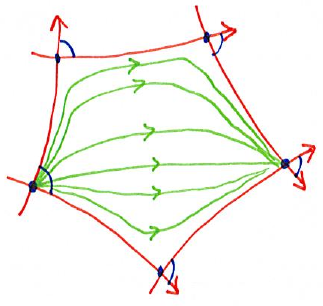}}
\put(0,52){\includegraphics[width=.35\textwidth]{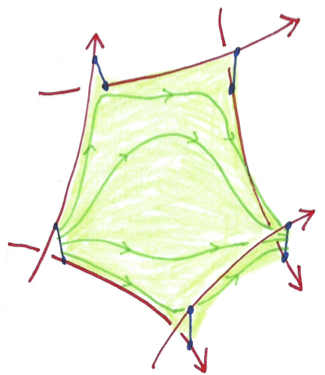}}
\put(65,5){\includegraphics[width=.45\textwidth]{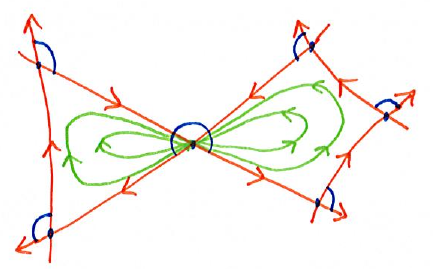}}
\put(65,60){\includegraphics[width=.45\textwidth]{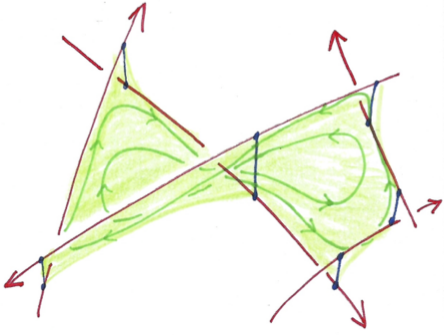}}
\put(94,20){$a$}
\put(75,15){$P^-$}
\put(113,27){$P^+$}
\end{picture}
\caption{A polygon~$P$ (bottom left), a vector field~$v$ on~$P$, and the lift~$H(P, v)$ in the unitary tangent bundle (top left). 
Also, two polygons~$P^-, P^+$ that form a butterfly with common vertex~$a$ (bottom right), and the associated butterfly surface~$B(P^-, P^+, a)$ in the unitary tangent bundle (top right).}
\label{F:Horizontal}
\end{figure}
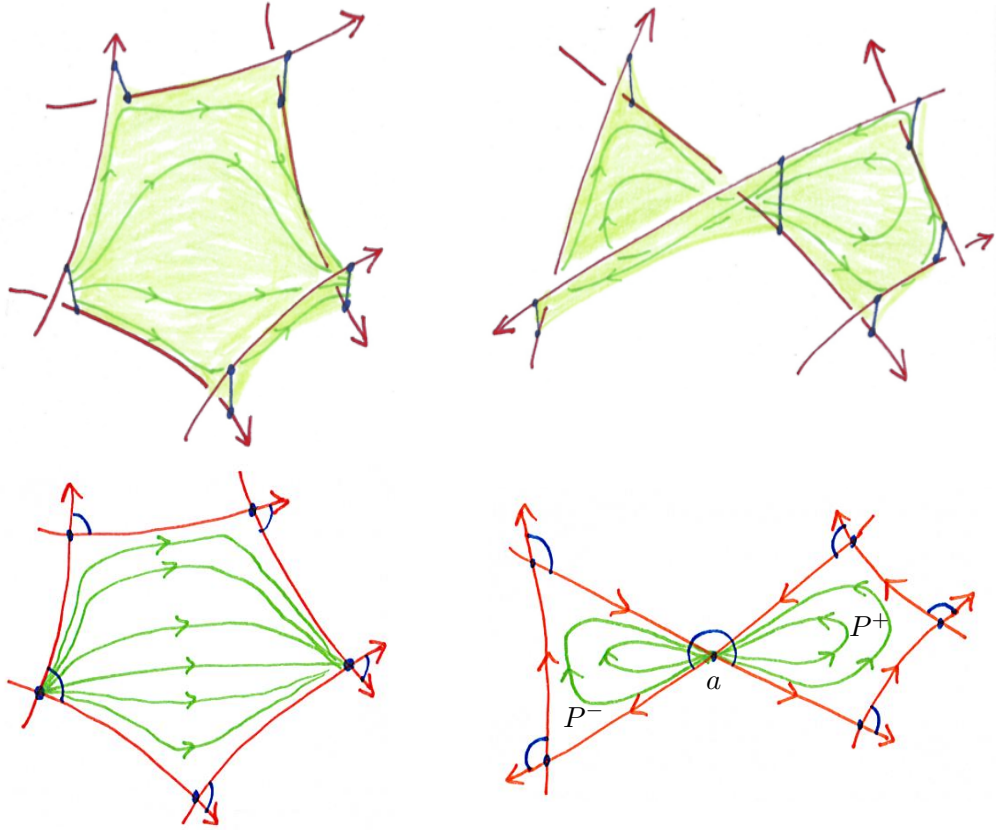

\begin{definition}(see Figure~\ref{F:Horizontal} right)
Given a hyperbolic 2-orbifold~$\OO$, two polygons with geodesic boundary~$P^+, P^-$ in~$\OO$ that have one vertex~$a$ in common, equip the interior of $P^+$ with a vector field~$v^+$ all of whose integral lines go from $a$ to $a$ and turn to the left, and equip the interior of $P^-$ with a vector field~$v^-$ all of whose integral lines go from $a$ to $a$ and turn to the right. 
The union~$H(P^-,v^-)\cup H(P^+,v^+)$ is called the~\term{butterfly surface} associated to the triple~$(P^-, P^+, a)$, and is denoted by~$B(P^-, P^+,a)$. The two subsurfaces $H(P^-,v^-)$, $H(P^+,v^+)$ are called the \term{wings} of the butterfly. 
\end{definition}

In the above definition, since the integral lines of~$v^+$ in~$P^+$ have no inflection point, the surface~$H(P^+, v^+)$ is transverse to the geodesic flow~$\fgeod$. 
The same holds for~$H(P^-, v^-)$. 
Also the vertical boundaries of~$H(P^+, v^+)$ and~$H(P^-, v^-)$ in~$a$ coincide with opposite orientation, and therefore cancel each other out, so that the butterfly surface~$B(P^-, P^+,a)$ has horizontal boundary in the lifts of the edges of~$P^-\cup P^+$ and vertical boundary in the lifts of the vertices of~$P^-\cup P^+$ different from~$a$. 

Note that different vector fields $v^+, v^-$ satisfying the above assumptions yield isotopic surfaces with the same boundary in the unitary tangent bundle. 
Therefore the specific choice of vector fields does not matter as long as we are only concerned with isotopy classes of surfaces in the unitary tangent bundle. 
Butterfly surfaces were considered (in a more pecular context and not under this name) in~\cite{GenusOne, HM}.


\subsection{Proof of Theorem~\ref{T:G1BS} when $g=0$ and $p_1, \dots, p_n\ge 3$}\label{S:Sphere}

Fix an integer $n\ge 3$, and $p_1, \dots, p_n\ge 3$ so that $\frac1{p_1}+\dots+\frac1{p_n}<n{-}2$ (this is the condition for the considered orbifold to be hyperbolic). 

\subsubsection{Choice of the orbifold metric}\label{S:MetricSphere} (see Figure~\ref{F:CasSphere2}) 
Consider a $n$-gon~$E$ in~$\Hy$ with vertices by~$b_1, \dots, b_n$ so that the angle at~$b_i$ is $\pi/p_i$. 
Consider the image~$F$ of $E$ under a reflection across~$(b_1b_2)$, and denote its vertices by $b_1, b_2, b'_3, \dots, b'_n$. 
Identify every side $[b_ib_{i+1}]$ of ~$E$ with the side  $[b'_ib'_{i+1}]$ of~$F$. 
The quotient of $E\cup F$ under these identifications is a hyperbolic orientable 2-orbifold that is a sphere with $n$ conic points coming for the vertices~$b_1, \dots, b_n$. 
The total angle at~$b_i$ is~$2\pi/b_i$, so that the order of~$b_i$ is $p_i$. 
Denote the obtained orbifold by~$\OO_{0; p_1, \dots, p_n}$. 
The polygons $E$ and $F$ induce two hemispheres on~$\OO_{0; p_1, \dots, p_n}$, while the segments~$[b_1b_2], [b_2b_3],\dots, [b_nb_1]$ form an equator. 

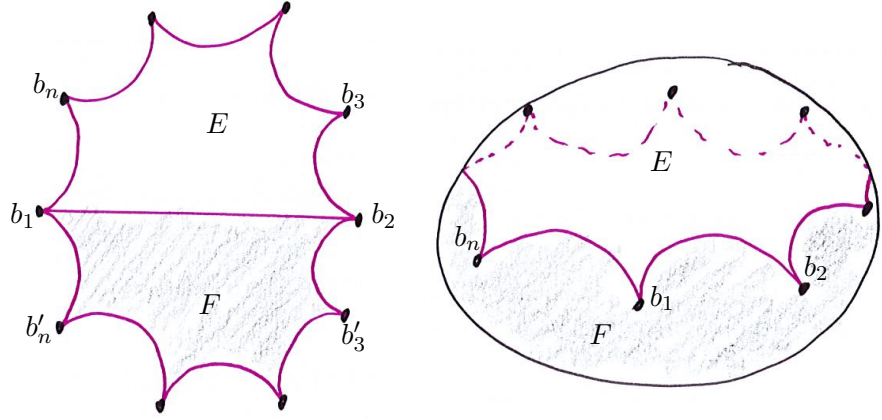
\begin{figure}[ht]
\begin{picture}(110,60)(0,0)
\put(0,0){\includegraphics[width=.32\textwidth]{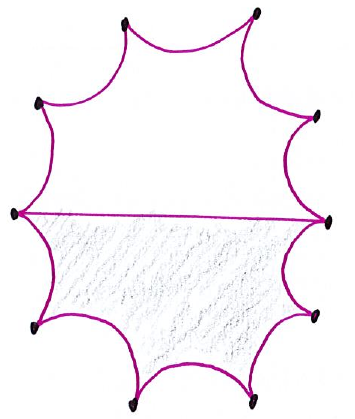}}
\put(53,2){\includegraphics[width=.42\textwidth]{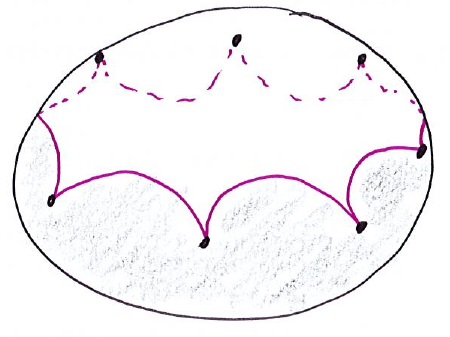}}
\put(24,38){$E$}
\put(23,14){$F$}
\put(-2,26){$b_1$}
\put(46,26){$b_2$}
\put(42,42){$b_3$}
\put(1,43.3){$b_n$}
\put(42,10){$b'_3$}
\put(0,11){$b'_n$}
\put(83,33){$E$}
\put(75,10){$F$}
\put(83,15){$b_1$}
\put(103.5,19){$b_2$}
\put(57,23){$b_n$}
\end{picture}
\caption{The orbifold~$\OO_{0; p_1, \dots, p_n}$, with the $n$ conic points~$b_1, \dots, b_n$ of respective order~$p_1, \dots, p_n$ on the equator. 
}
\label{F:CasSphere2}
\end{figure}

\subsubsection{Choice of the boundary orbits}\label{S:BoundarySphere}
For $i=1, \dots, n$, denote by $r_i$ the rotation around~$b_i$ of angle~$+2\pi/p_i$. 
The group~$G$ generated by the family~$\{r_i\}_{i=1, \dots, n}$ is the orbifold fundamental group of~$\OO_{0; p_1, \dots, p_n}$, that is, one can identify $\OO_{0; p_1, \dots, p_n}$ with the quotient $\Hy/G$. 
The composition $r_{i}\circ r_{i+1}^{-1}$ (counting mod~$n$) is a hyperbolic isometry of~$\Hy$ that we denote by~$g_i$ and whose oriented axis we denote by~$\beta_i$, see Figure~\ref{F:Ell33} left. 
The projection of~$\beta_i$ on~$\OO_{0;p_1, \dots, p_n}$ is a figure-eight curve (also denoted by~$\beta_i$) that winds positively around~$b_{i+1}$ and negatively around~$b_i$, see Figure~\ref{F:Ell33} right.
Denote by~$\vec\beta$ the link~$\vec\beta_1\cup\dots\cup\vec\beta_n$ in~$\U\OO_{0; p_1, \dots, p_n}$. 

\begin{figure}[ht]
\begin{picture}(140,65)(0,0)
\put(0,0){\includegraphics[width=.45\textwidth]{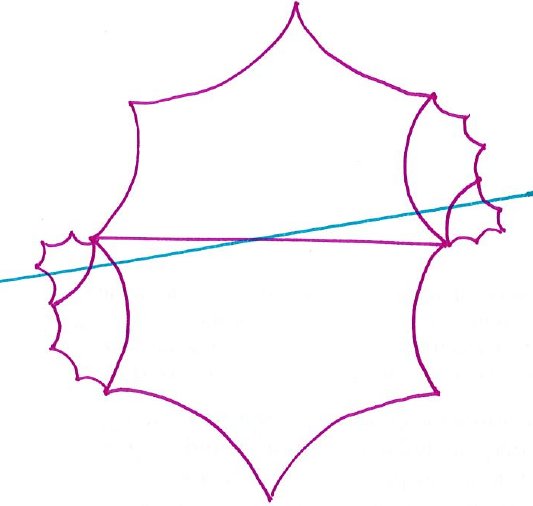}}
\put(81,15){\includegraphics[width=.32\textwidth]{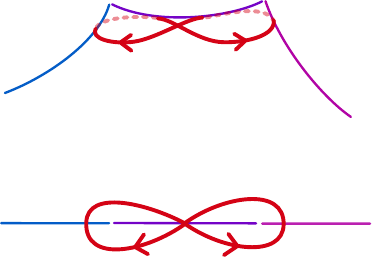}} 
\put(32,42){$E$}
\put(32,20){$F$}
\put(57,35.5){\Small{$r_2(E)$}}
\put(51.7,42.5){\small{$r_2(F)$}}
\put(8,22){\Small{$r_1(E)$}}
\put(4,30){\Small{$r_1(F)$}}
\put(10,36){$b_1$}
\put(57,29){$b_2$}
\put(-3,23){\small{$r_1(b_2)$}}
\put(54,54){$b_3$}
\put(62,41){\small{$r_2(b_1)$}}
\put(94,49){$b_1$}
\put(115,49){$b_2$}
\put(102,42){$m_1$}
\put(102,22){$m_1$}
\put(91,14){$\beta_1$}
\put(116,14){$\beta_1$}
\end{picture}
\caption{On the left the polygons~$E, F$ and their images under $r_1$ and~$r_2$. 
The isometry that sends $r_2(E)$ onto $r_1(E)$ (and respecting the labels of the vertices) is hyperbolic and that its axis is the blue line. 
On the right, its projection on the orbifold of the axis is a figure-eight curve that winds around $b_1$ and $b_2$, with a double point~$m_1$ on the equator.}
\label{F:Ell33}
\end{figure}
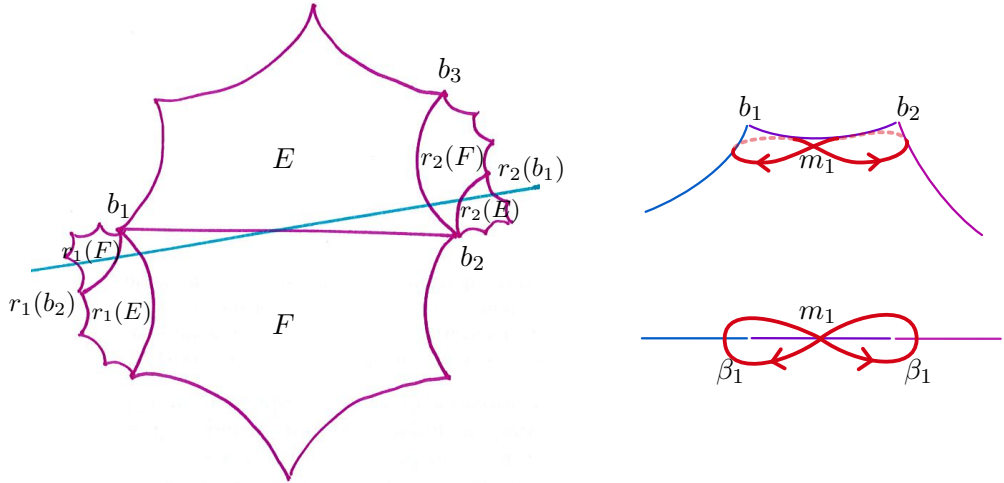

\subsubsection{Choice of the surface in~$\U\OO_{0; p1, \dots, p_n}$}\label{S:SurfaceSphere}
Fix $i$ between $1$ and $n$. 
Call $m_i$ the double-point on~$\beta_i$. 
Then $\beta_i$ intersects $\beta_{i-1}$ and $\beta_{i+1}$ twice in a neighbourhood of~$b_i$ (at points called $c_i^E$ and $c_i^F$ depending on the hemisphere in which they lie) and $b_{i+1}$ (at points called $c_{i+1}^E$ and $c_{i+1}^F$), see Figure~\ref{F:S33} left. 
Therefore they define two triangles~$T_i^-$ with vertices $m_i, c_i^E$ and $c_i^F$, and~$T_i^+$ with vertices $m_i, c_{i+1}^E$ and $c_{i+1}^F$. 
Thanks to the orientations of the edges, there is a butterfly surface $B(T_i^-, T_i^+, m_i)$, constructed with one vector field~$v_i^{-}$ on~$T_i^{-}$ so that the integral curves of~$v_i^{-}$ turn to the right (with non-vanishing curvature) and go from $m_i$ to~$m_i$, as in Figure~\ref{F:S33} right, and one vector field $v_i^{+}$ on~$T_i^{+\,\circ}$ so that the integral curves of~$v_i^{+}$ turn to the left (with non-vanishing curvature) and go from $m_i$ to~$m_i$. 
The horizontal boundary of~$B(T_i^-, T_i^+, m_i)$ consists of the lifts of the sides of $T_i^{+}$ and~$T_i^-$ travelled with anti-clockwise orientation. 
Its vertical boundary consists of four segments in the fibers of $c_{i}^E, c_{i}^F, c_{i+1}^E$, and $c_{i+1}^F$.

\begin{figure}[ht]
\begin{picture}(160,27)(0,0)
\put(0,-1){\includegraphics[width=.42\textwidth]{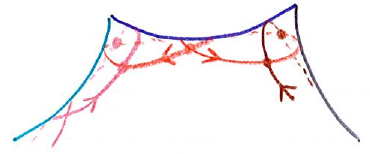}}
\put(68,2){\includegraphics[width=.56\textwidth]{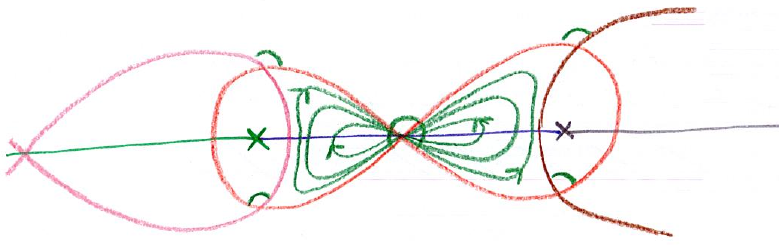}}
\put(15,23){$b_i$}
\put(30,21){$m_i$}
\put(20,10){$c_i^F$}
\put(41,10){$c_{i+1}^F$}
\put(50,25){$b_{i+1}$}
\put(112,19){$m_i$}
\put(95,25){$c_i^E$}
\put(95,2){$c_i^F$}
\put(106,6){$T_i^{-}$}
\put(116,5){$T_i^{+}$}
\put(128,27){$c_{i+1}^E$}
\put(125,3){$c_{i+1}^F$}
\end{picture}
\caption{On the left, the geodesics $\beta_i$ (red), $\beta_{i-1}$ (pink) and $\beta_{i+1}$ (brown). 
On the right, seen from above, the same geodesics, the triangles $T_i^{-}$ and $T_i^{+}$, and the integral curves of the vector fields $v_i^{-}$ and $v_i^{+}$ (in green). 
When lifted in~$\U\OO_{0;p_1, \dots, p_n}$ these vector fields lift into a butterfly surface~$B(T_i^-, T_i^+, m_i)$. 
The vertical boundaries are represented by arcs of direction at the points $m_i, c_{i}^E$, $c_{i}^F, c_{i+1}^E$, and $c_{i+1}^F$.}
\label{F:S33}
\end{figure}
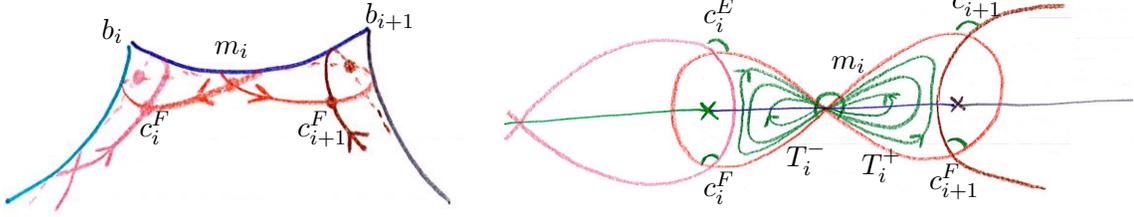

Consider the surface~$S_{0; p_1, \dots, p_n}$ that is the union of the $n$ butterfly surfaces $B(T_1^-,T_1^+,m_1)$ $\cup\dots\cup B(T_n^-,T_n^+,m_n)$. 
Its horizontal boundary is exactly~$\vec\beta$ with negative orientation, see Figure~\ref{F:S33} right. 
For every point of type~$m_i$, $c_i^E$, or $c_i^F$, the vertical boundary components in the corresponding fiber come from two adjacent triangles and pairwise cancel, so that $S_{0; p_1, \dots, p_n}$ has no vertical boundary. 
Unlike the vertical surfaces of the previous section, the surface $S_{0; p_1, \dots, p_n}$ is already smooth. 
Also, it is transverse to the geodesic flow, thanks to the assumptions on the orbits of the vector fields~$v_i^{-}$ and~$v_i^{+}$.

\subsubsection{Computation of the genus}\label{S:GenusSphere}
We want to compute the Euler characteristic of the surface~$S_{0; p_1, \dots, p_n}$. 
It consists of the $2n$ wings of the considered butterflies, that are horizontal surfaces of the form $H(T, v)$. 
Each of these parts is a hexagon with $3$ horizontal sides and $3$ vertical sides that are glued with an adjacent hexagon. 
The three horizontal sides contribute by $-1$ to the Euler characteristic, and the three vertical sides contribute by~$-\frac12$ (as they are counted twice). 
In the same vein, all six vertices contribute by~$+\frac12$. 
Therefore the contribution of every hexagon is~$+1-3\cdot 1-3\cdot\frac12+6\cdot \frac12=-\frac12$.

Since there are $2n$ such hexagons, we find $\chi(S_{0; p_1, \dots, p_n})=2n\cdot-\frac12=-n$. 
Since it has $n$ boundary components, its genus is~$1$. 

\medskip

Alternatively, as in Remark~\ref{R:Genus}, one can use that the surface~$S_{0; p_1, \dots, p_n}$ is transverse to the geodesic flow. 
Therefore, using the Poincaré-Hopf formula, each boundary component contributes to the Euler characteristic an amount equal to minus its absolute self-linking with respect to the stable foliation~$-|\slk^{\Fs}(\vec\beta_i)|$. 
By rotating the framing toward the vertical direction, $\slk^{\Fs}(\vec\beta_i)$ is equal to the self-linking with respect to the fiber direction $\slk^{\U}(\vec\beta_i)$. 
One now checks that around a point of type~$c_i^{E/F}$, the contribution is~$-1/2$, while around a of type~$m_i$ the contribution is~$+1/2$, see the top right of Figure~\ref{F:Horizontal}. 
Since each curve~$\beta_i$ contains four points of the first type, and two of the second (it goes through~$m_i$ twice), $-|\slk^{\U}(\vec\beta_i)|$ is equal to~$-1$.  
Therefore $S_{0; p_1, \dots, p_n}$ has Euler characteristic~$-n$, hence it is a $n$-punctured torus.

\subsubsection{Intersection with orbits of the geodesic flow}\label{S:IntersectionSphere}
As explained in Section~\ref{S:Butterfly}, butterfly surfaces are transverse to the geodesic flow~$\fgeod$, and so is $S_{0; p_1, \dots, p_n}$. 

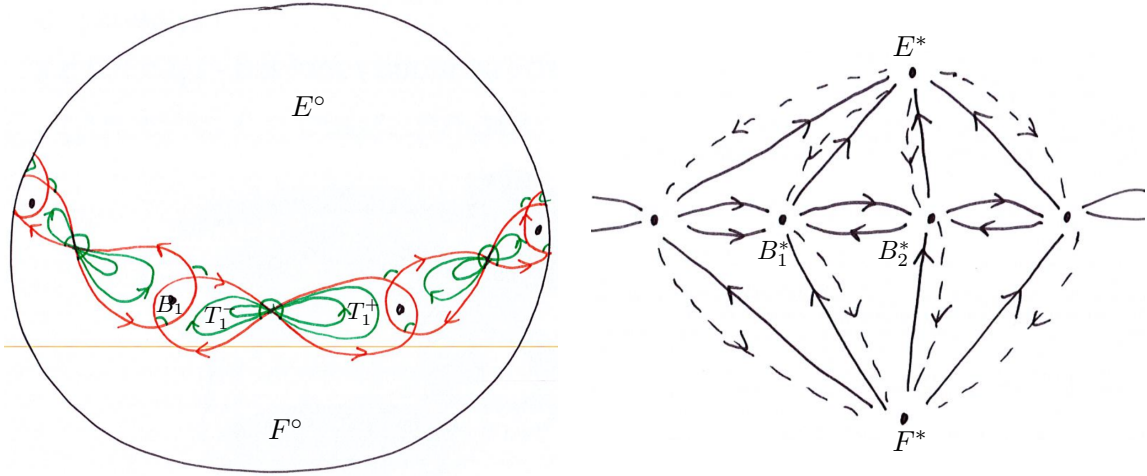
\begin{figure}[ht]
\begin{picture}(160,70)(0,0)
\put(0,0){\includegraphics[width=.49\textwidth]{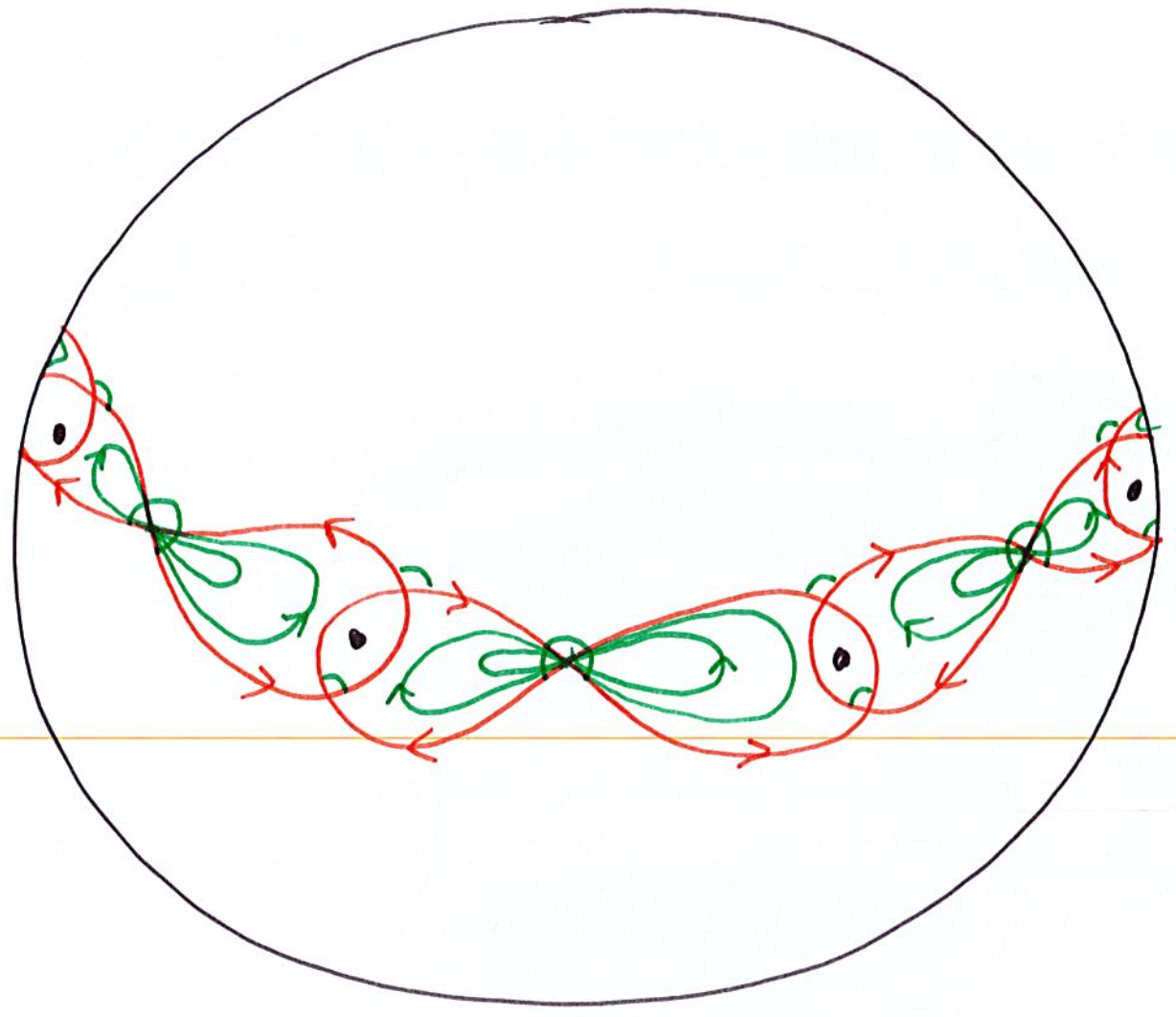}}
\put(78,5){\includegraphics[width=.49\textwidth]{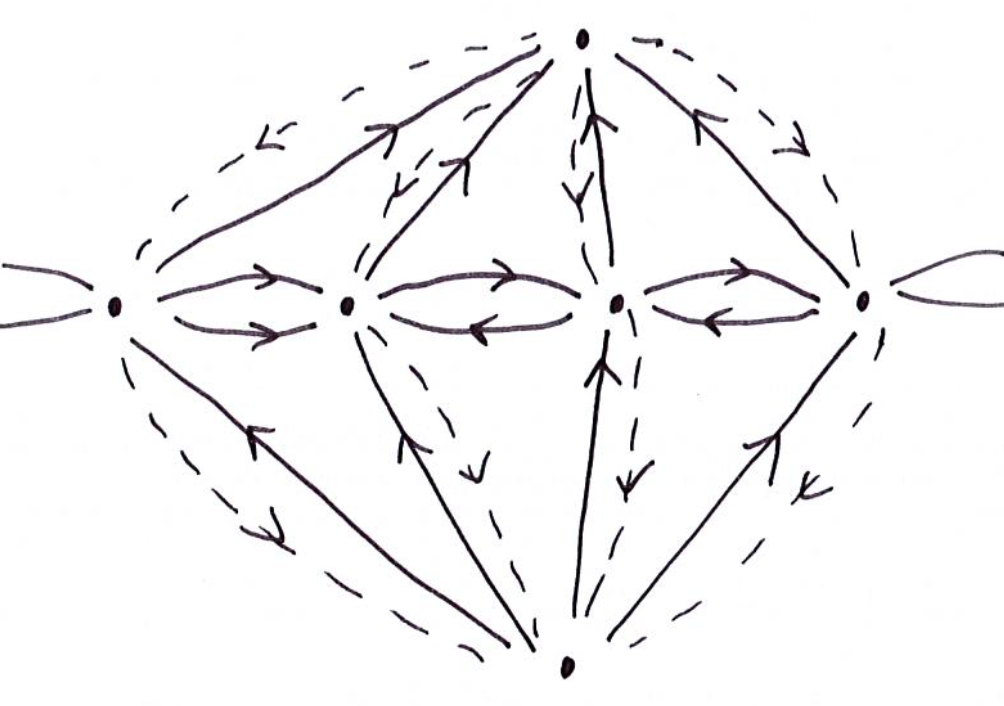}}
\put(38,48){$E^\circ$}
\put(35,5){$F^\circ$}
\put(20,22){\Small{$B_1$}}
\put(26.5,20.5){\Small{$T_1^-$}}
\put(45.5,21){\Small{$T_1^+$}}
\put(118,56){$E^*$}
\put(118,3.5){$F^*$}
\put(100,29){\small{$B_1^*$}}
\put(116,29){\small{$B_2^*$}}
\end{picture}
\caption{On the left the surface~$S_{0; p_1, \dots, p_n}$ and the different regions of the orbifold delimited by the multicurve~$\beta$. On the right the dual graph~$G_{0; p_1, \dots, p_n}^*$.}
\label{F:S3333}
\end{figure}

By construction, the orbifold~$\OO_{0;p_1, \dots, p_n}$ consists of two hemispheres glued together. 
The collection~$\beta$ lies in a neighborhood of the equator $[b_1b_2]\cup[b_2b_3]\cup\dots[b_nb_1]$. 
We call~$E^\circ$ and $F^\circ$ the two components of $\OO_{0;p_1, \dots, p_n}\setminus\beta$ that do not intersect the equator. 
We call~$B_i$ the bigon that contains the conic point~$b_i$, its vertices are $c_i^+, c_i^-$. 
In this way~$E^\circ, F^\circ, B_1, \dots, B_n, T_1^-, T_1^+, \dots,  T_n^-, T_n^+$ tesselate the orbifold~$\OO_{0; p_1, \dots, p_n}$. 

Define the graph~$G^*_{0; p_1, \dots, p_n}$ with $n+2$ vertices labeled~$E^*, F^*, B_1^*, \dots, B_n^*$ and whose edges are shown on Figure~\ref{F:S3333}. Some edges are plain and some are dotted as on the picture. 

Consider an arbitrary oriented geodesic. 
Since all pieces of the tesselation of~$\gamma$ on $\OO_{0;p_1, \dots, p_n}$ are topological disc containing at most one conic points, $\gamma$ can only stay a uniformly bounded time in each of them. 
Moreover, unless~$\gamma$ is one of $\beta_1, \dots, \beta_n$, no more than two triangles appear consecutively, so that every long enough subarc of $\gamma$ visits one of the regions $E^\circ, F^\circ, B_1, \dots, B_n$. 
Therefore~$\gamma$ can be decomposed in the form $\dots \cup\gamma_{-1}\cup \gamma_{0}\cup \gamma_{1}\cup\dots$, where each subsegment $\gamma_j$ crosses one of the regions $E^\circ, F^\circ, B_1, \dots, B_n$ and has bounded length. 
Thus one can associate to~$\gamma$ an infinite path~$\gamma^*$ in~$G^*_{0; p_1, \dots, p_n}$ that corresponds to the list of those faces visited by~$\gamma$. 

Now, every bold edge in~$G^*_{0; p_1, \dots, p_n}$ corresponds to a positive intersection between the lift~$\vec\gamma$ with~$S_{0; p_1, \dots, p_n}$.  
Since $\gamma^*$ cannot avoid the bold edges (indeed, at least every third edge followed by~$\gamma^*$ is bold), $\gamma$ necessarily intersects~$S_{0; p_1, \dots, p_n}$ after a uniformly bounded time, so that~$S_{0; p_1, \dots, p_n}$ is indeed a genus-one Birkhoff section for~$\fgeod$.


\subsection{Proof of Theorem~\ref{T:G1BS} when $g\ge 1$, $n=2g+2$ and $p_1, \dots, p_n\ge 3$}\label{S:Butterfly1}

Fix an integer $g\ge 1$, set $n:=2g+2$, and fix integers $p_1, \dots, p_n\ge 3$. 
The construction we now describe is close to the construction in the previous section. 
Actually, if every $p_i$ was even, it could be obtained by an order 2 covering of the latter. 

\subsubsection{Choice of a suitable orbifold metric}\label{S:MetricButterfly}
The surface is similar as the one in Section~\ref{S:MetricSurface}, but with conic points at the vertices of the considered $2g{+}2$ gons, as depicted on Figure~\ref{F:CasButterfly}. 

Start with a $2g{+}2$-gon~$A$ in~$\Hy$, whose vertices are denoted by~$a_1, \dots, a_{2g+2}$, so that the angle in~$a_i$ is ${\pi}/{2p_i}$. 
Consider the symmetry~$s_h$ with axis~$(a_1a_2)$ and denote by~$B$ the image of~$A$. 
Also consider the symmetry~$s_v$ with axis~$(a_1a_{2g+2})$ and denote by~$D$ the image of~$A$. 
Finally consider the rotation~$r$ with center $a_1$ and angle $-{\pi}/{p_1}$, and denote by~$C$ the image of~$A$. 

Consider the following identifications of the sides of~$A\cup B\cup C\cup D$:
\begin{itemize}
\item $[a_{2i}a_{2i+1}]$ with $s_{v}([a_{2i}a_{2i+1}])$, 
\item $r([a_{2i}a_{2i+1}])$ with $s_{h}([a_{2i}a_{2i+1}])$, 
\item $[a_{2i+1}a_{2i+2}]$ with $s_{h}([a_{2i+1}a_{2i+2}])$, 
\item $r([a_{2i+1}a_{2i+2}])$ with $s_{v}([a_{2i_1}a_{2i+2}])$. 
\end{itemize}
 In other words, edges of $A$ are identified with their images under $s_{v}$ and $s_h$ in $B$ and $D$ alternatively. 

The quotient is a hyperbolic orbifold, that we denote by~$\Sigma_{g; p_1, \dots, p_{2g+2}}$. 
The genus of~$\Sigma_{g; p_1, \dots, p_{2g+2}}$ is~$g$ (as in Section~\ref{S:Surface}). 
Since every vertex~$a_i$ is identified with $s_v(a_i), s_h(a_i)$ and $r(a_i)$, in the quotient, the total angle at~$a_i$ is~$\frac{2\pi}{p_i}$, hence $a_i$ is a conic point of order~$p_i$. 
The points $a_1, \dots, a_{2g+2}$ projects onto $2g+2$ conic points of respective orders~$p_1, \dots, p_{2g+2}$. 

Also we denote by $e_i$ the projection of the edge~$[a_ia_{i+1}]$ in the quotient and by $e'_i$ the projection of the edge~$[r(a_i)r(a_{i+1})]$. 
The faces, sides and vertices of~$A\cup B\cup C\cup D$ then induce a graph~$G_g$ on~$\Sigma_g$ with~$4$ faces (that correspond to $A$, $B$, $C$, and $D$), $4g{+}4$ edges (namely $e_1, e'_1, \dots, e_{2g+2}, e'_{2g+2}$), and $2g{+}2$ vertices. 
Picturing~$\Sigma_{g; p_1, \dots, p_{2g+2}}$ as on Figure~\ref{F:CasButterfly2}, the two faces~$A$ and $B$ lie on the top, and~$C$ and $D$ lie on the bottom. 
Also~$B$ and $C$ lie on the front, and~$A$ and $D$ lie on the back. 

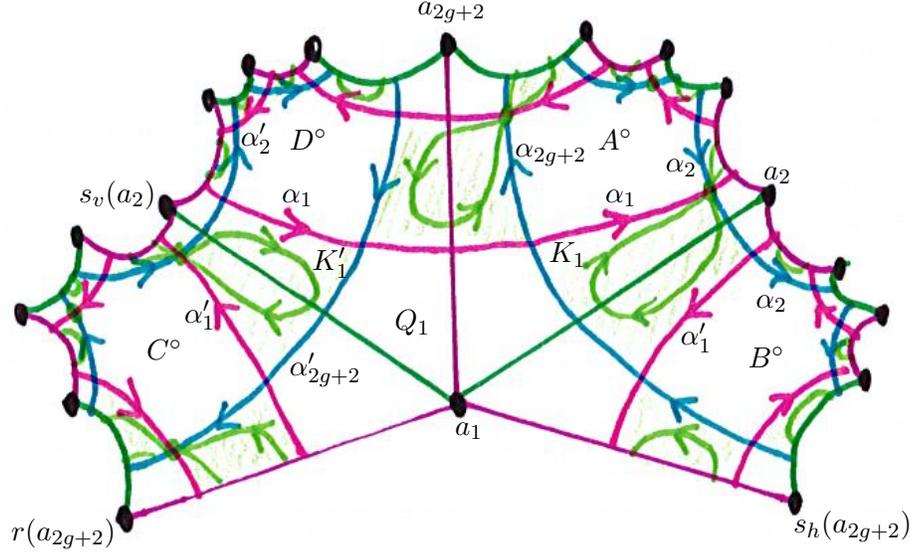
\begin{figure}[ht]
\begin{picture}(120,70)(0,0)
\put(0,0){\includegraphics[width=.8\textwidth]{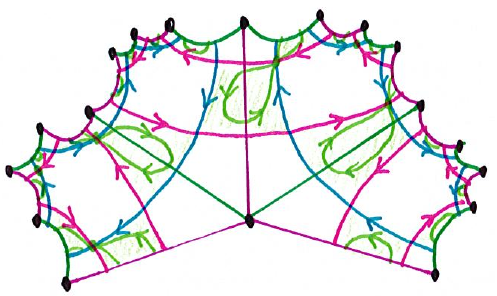}}
\put(79,52){$A^\circ$}
\put(99,23){$B^\circ$}
\put(19,24){$C^\circ$}
\put(38,52){$D^\circ$}
\put(60,14){$a_1$}
\put(101,48){$a_2$}
\put(10,45){$s_v(a_2)$}
\put(55,70){$a_{2g+2}$}
\put(105,1){$s_h(a_{2g+2})$}
\put(1,0){$r(a_{2g+2})$}
\put(37,45){$\alpha_1$}
\put(80,45){$\alpha_1$}
\put(88,49){$\alpha_2$}
\put(100,31){$\alpha_2$}
\put(24,29){$\alpha'_1$}
\put(90,26){$\alpha'_1$}
\put(68,51){$\alpha_{2g+2}$}
\put(31.5,52.5){$\alpha'_2$}
\put(38,22){$\alpha'_{2g+2}$}
\put(52,28){$Q_1$}
\put(72.5,37){$K_1$}
\put(41,36){$K'_1$}
\end{picture}
\caption{The $2g{+}2$-gon~$A$, and its images~$B, D, C$ under symmetries and the rotation~$r$ around~$a_1$. 
The geodesics~$\alpha_1, \alpha_3, \dots, \alpha_{2g+1}$ visit $A$ and $D$, they are shown in blue. 
Their images~$\alpha'_1, \alpha'_3, \dots, \alpha'_{2g+1}$ under~$r$ visit $B$ and $C$, they are also shown in blue. 
The geodesics~$\alpha_2, \alpha_4, \dots, \alpha_{2g+2}$ visit $A$ and $B$, they are shown in pink. 
Their images~$\alpha'_2, \alpha'_4, \dots, \alpha'_{2g+2}$ under~$r$ visit $C$ and $D$, they are also shown in pink.
The capital letters denote the faces of the complement of the collection~$(\alpha_i)_{i=1, \dots, 2g+2}\cup(\alpha'_i)_{i=1, \dots, 2g+2}$.  
}
\label{F:CasButterfly}
\end{figure}

\subsubsection{Choice of the boundary orbits}\label{S:BoundaryButterfly}

Unlike the construction of Section~\ref{S:Surface}, we do not choose the edges~$[a_1a_2], [a_2a_3], \dots$ for the boundary of the Birkhoff section, as they visit the conic points. 
Instead we consider curves obtained by pushing these edges away from the conic points. 

For every $i$ in $\{0, \dots, g\}$, consider the oriented arc~$e_{2i+1}$ in~$A$ that starts in $[a_{2i}a_{2i+1}]$ and is orthogonal to it, and ends in~$[a_{2i+2}a_{2i+3}]$ and is orthogonal to it. 
Also consider the oriented arc~$e_{2i+2}$ in~$A$ that starts in $[a_{2i+3}a_{2i+4}]$ and is orthogonal to it, and ends in~$[a_{2i+1}a_{2i+2}]$ and is orthogonal to it. 
These arcs extend into the faces~$B$ and $D$ respectively, and define oriented geodesics $(\alpha_i)_{1\le i\le 2g+2}$ on~$\Sigma_{g; p_1, \dots, p_{2g+2}}$, see Figure~\ref{F:CasButterfly}.

The same construction in~$C$ (or composing the previous paragraph with the rotation~$r$) yields oriented geodesics $(\alpha'_i)_{1\le i\le 2g+2}$ on~$\Sigma_{g; p_1, \dots, p_{2g+2}}$.

Then consider the link~$L_{g; p_1, \dots, p_{2g+2}}:=\vec\alpha_1\cup\dots\cup\vec\alpha_{2g+2}\cup\vec\alpha'_1\cup\dots\cup\vec\alpha'_{2g+2}$ in~$\U\Sigma_{g; p_1, \dots, p_{2g+2}}$.

\subsubsection{Choice of the surface}\label{S:SurfaceButterfly}

The collection~$(\alpha_i)_{1\le i\le 2g+2}\cup(\alpha'_i)_{1\le i\le 2g+2}$ decomposes the orbifold $\Sigma_{g; p_1, \dots, p_{2g+2}}$ into $4g+6$ regions, see Figures~\ref{F:CasButterfly} and~\ref{F:CasButterfly2}, namely 
\begin{itemize}
\item $2g{+}2$ quadrilaterals containing the points $a_1, \dots, a_{2g+2}$, that we denote by $Q_1$, $\dots, Q_{2g+2}$, 
\item $4g{+}4$ quadrilaterals containing the central parts of the edges $[a_ia_{i+1}]$ and $[r(a_i)r(a_{i+1})]$, that we denote by $K_1, \dots, K_{2g+2}$ and $K'_1, \dots, K'_{2g+2}$ respectively, 
\item four faces corresponding the centers of $A, B, C,$ and~$D$, that we denote by~$A^\circ, B^\circ, C^\circ,$ and~$D^\circ$. 
\end{itemize}

\begin{figure}[ht]
\begin{picture}(120,90)(0,0)
\put(0,0){\includegraphics[width=.8\textwidth]{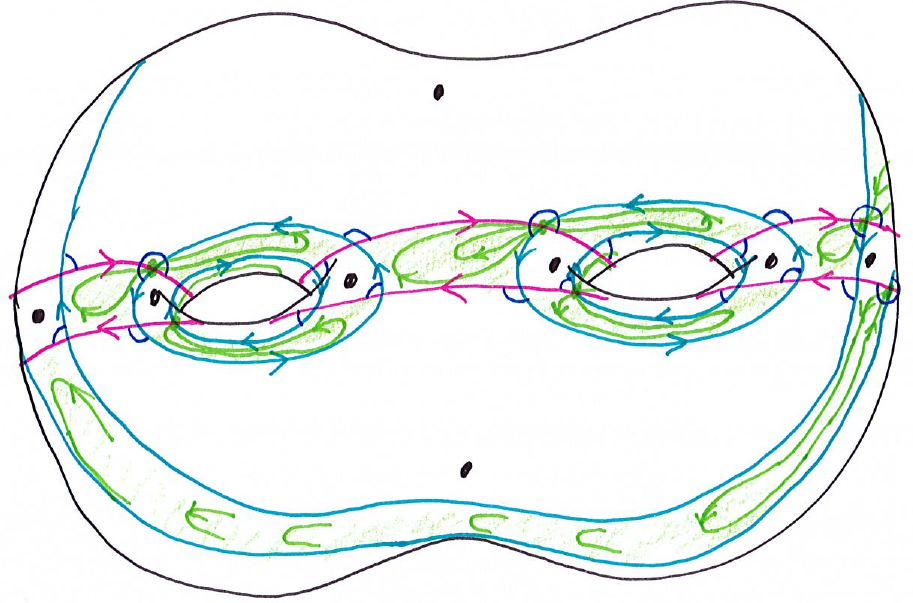}}
\put(70,65){$A^\circ$}
\put(75,17){$B^\circ$}
\put(-3,35){$Q_1$}
\put(10,38){$K_1$}
\put(29,44){$K_2$}
\put(36,32){$K'_2$}
\put(55,342){$K_3$}
\put(16,15){$K'_{2g+2}$}
\put(118,44){$Q_{2g+2}$}
\put(12,48){$\alpha_1$}
\put(28,51){$\alpha_2$}
\put(53,51){$\alpha_3$}
\put(13,31){$\alpha'_1$}
\put(54,38){$\alpha'_3$}
\put(28,29){$\alpha_2$}
\put(75,33){$\alpha_4$}
\end{picture}
\caption{The orbifold $\Sigma_{g; p_1, \dots, p_{2g+2}}$ with the geodesics $\alpha_1,\dots,\alpha_{2g+2},\alpha'_1,\dots,\alpha'_{2g+2}$ that form a tesselations into regions $A^\circ,  B^\circ, C^\circ, D^\circ, (Q_i)_{1\le i\le 2g+2}, (K_i)_{1\le i\le 2g+2}, (K'_i)_{1\le i\le 2g+2}$. 
The butterfly surfaces~$B(K_1, K_2, s_1), \dots, B(K_{2g+1}, K_{2g+2}, s_{2g+1})$, $B(K'_1, K'_2, s'_1), \dots, B(K'_{2g+1}, K'_{2g+2}, s'_{2g+1})$ are shown in green. 
Their union is the Birkhoff section~$S_{g; p_1, \dots, p_{2g+2}}$. 
}
\label{F:CasButterfly2}
\end{figure}
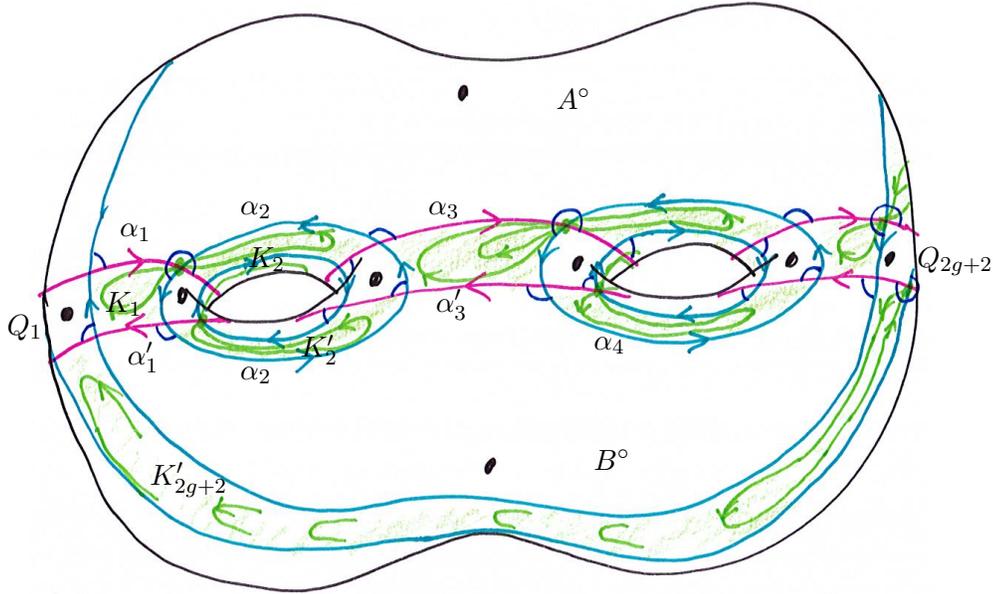

For $0\le i\le g$, the quadrilaterals $K_{2i+1}$ and $K_{2i+2}$ meet at a single vertex that we call~$s_{2i+1}$, and induce a butterfly as defined in Section~\ref{S:Butterfly}. 
We then consider the butterfly surface~$B(K_{2i+1}, K_{2i+2}, s_{2i+1})$. 
Similarly, $K'_{2i+1}$ and $K'_{2i+2}$ meet at a single vertex that we call~$s'_{2i+1}$, and induce a butterfly. 
We then consider the butterfly surface~$B(K'_{2i+1}, K'_{2i+2}, s'_{2i+1})$. 

Consider the union of all these butterflies $B(K_1, K_2, s_1)\cup B(K'_1, K'_2, s'_1)\cup B(K_3, K_4, s_3)\cup\dots\cup B(K'_{2g+1}, K'_{2g+2}, s_{2g+1})$ and denote it by~$S_{g; p_1, \dots, p_{2g+2}}$. 
It is a (horizontal) surface whose vertical boundary is empty and whose horizontal boundary is exactly the link~$L_{g; p_1, \dots, p_{2g+2}}$. 

\subsubsection{Computation of the genus}\label{S:GenusButterfly}

The surface~$S_{g; p_1, \dots, p_{2g+2}}$ is made of $2g+2$ butterflies. 
Each butterfly is made of two wings that are the lifts of vector fields on quadrilaterals. 
As in Section~\ref{S:GenusSphere}, each such wing contributes by~$-1$ to the Euler characteristic, so that $\chi(S_{g; p_1, \dots, p_{2g+2}})=-4g-4$. 
Since the boundary is the link~$L_{g; p_1, \dots, p_{2g+2}}$ which has $4g+4$ components, the genus of~$S_{g; p_1, \dots, p_{2g+2}}$ is 1. 

\subsubsection{Intersection with orbits of the geodesic flow}\label{S:IntersectionButterfly}

We have to prove that the surface $S_{g; p_1, \dots, p_{2g+2}}$ is a Birkhoff section for the geodesic flow on~$\U\Sigma_{g; p_1, \dots, p_{2g+2}}$. 
Since it is a butterfly surface, its interior is transverse to the geodesic flow, and its boundary is tangent to it. 

\begin{figure}[ht]
\begin{picture}(100,80)(0,0)
\put(0,0){\includegraphics[width=.65\textwidth]{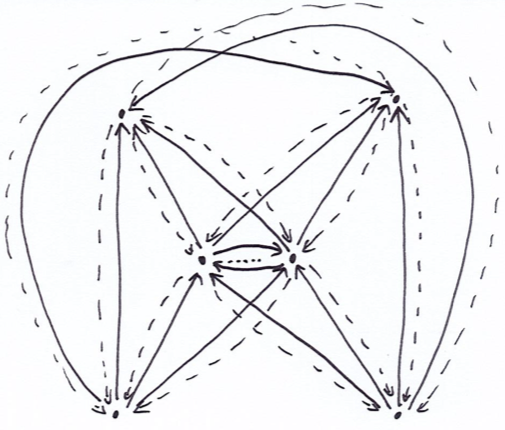}}
\put(19,61){$A^*$}
\put(78,63){$C^*$}
\put(20,-1){$B^*$}
\put(78,0){$D^*$}
\put(30,32){$Q_i^*$}
\put(60,32){$Q_j^*$}
\end{picture}
\caption{The dual graph~$G^*_{g; p_1, \dots, p_n}$. For simplicity we only draw two vertices among $Q^*_1, \dots, Q^*_{2g+2}$. Every long enough path in this graph must follow a bold edge, which means that the corresponding orbit of the geodesic flow intersects~$S_{g; p_1, \dots, p_n}$. 
}
\label{F:Sg333Dual}
\end{figure}
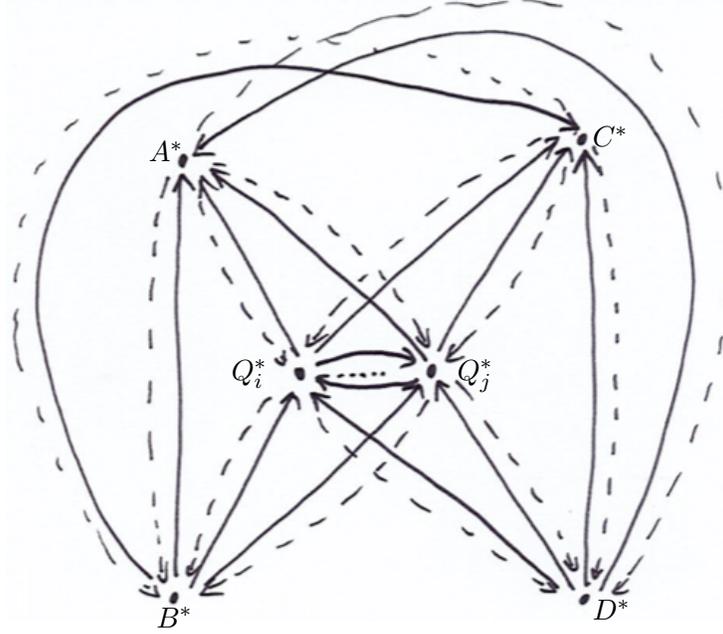

Define the graph~$G^*_{g; p_1, \dots, p_n}$ with $2g+6$ vertices labeled~$A^*, B^*, C^*, D^*, Q_1^*, \dots, Q_{2g+2}^*$ and whose edges are shown on Figure~\ref{F:Sg333Dual}. 
Some edges are dotted and some are plain according to the picture. 

Consider an arbitrary oriented geodesic~$\gamma$ on~$\OO_{g;p_1, \dots, p_n}$. 
Since all pieces of the decomposition are topological discs containing at most one conic points, $\gamma$ can only stay a bounded time in each of them. 
Moreover, unless~$\gamma$ is one of $\alpha_1, \dots, \alpha'_{2g+2}$ or is a union of arcs crossing each $K_i$ diagonally, it intersects one of the regions $A^\circ, B^\circ, C^\circ, D^\circ, Q_1, \dots, Q_{2g+2}$ after a uniformly bounded time.  

In this case~$\gamma$ can be decomposed in the form $\dots \cup\gamma_{-1}\cup \gamma_{0}\cup \gamma_{1}\cup\dots$, where each subsegment $\gamma_j$ crosses one of the regions $A^\circ, B^\circ, C^\circ, D^\circ, Q_1, \dots, Q_{2g+2}$ and has bounded length. 
Thus one can associate to~$\gamma$ an infinite path~$\gamma^*$ in $G^*_{g; p_1, \dots, p_n}$ that corresponds to the list of those faces visited by~$\gamma$. 
Now every bold edge in~$G^*_{g; p_1, \dots, p_n}$ corresponds to a positive intersection between the lift~$\vec\gamma$ with~$S_{0; p_1, \dots, p_n}$.  
Since $\gamma^*$ cannot avoid the bold edges (indeed, at least every third edge followed by~$\gamma^*$ is bold), $\gamma$ necessarily intersects~$S_{0; p_1, \dots, p_n}$ after a uniformly bounded time, so that~$S_{0; p_1, \dots, p_n}$ is indeed a genus-one Birkhoff section for~$\fgeod$. 

There remains the case where $\gamma$ is a concatenation of diagonals of the quadrilaterals~$K_i$. 
In this case, either the diagonal has one vertex in a vertex $s_i$ that is the middle of a butterfly, and at this point the lift~$\vec \gamma$ intersects~$S_{g; p_1, \dots, p_n}$, or the diagonal does not visit the middle of a butterfly, in which case the lift intersects~$S_{g; p_1, \dots, p_n}$ somewhere in the middle of the wing. 
In both cases, $\vec\gamma$ intersects~$S_{g; p_1, \dots, p_n}$. 

All-in-all, we checked that every long enough arc of orbit of the geodesic flow intersects~$S_{g; p_1, \dots, p_n}$, which is therefore a Birkhoff section. 


\subsection{Proof of Theorem~\ref{T:G1BS} when $g\ge 2$, $2g+2<n<2g+6$ and $p_1, \dots, p_n\ge 3$}\label{S:Butterfly2}

Let us explain how to modify the construction of the previous section in order to deal with $2g+3, 2g+4, 2g+5$, or~$2g+6$ conic points. 
For this we only have to add one conic point in one, two, three, or four of the faces $A, B, C$, and $D$. 
Adding these conic points changes the orbifold metric, but does not affect the construction nor the genus of the surface~$S_{g; p_1, \dots, p_n}$. 
Also, since there is no more than one conic point in each of the regions~$A^\circ, B^\circ, C^\circ, D^\circ$, this does not change the fact that~$S_{g; p_1, \dots, p_n}$ intersects all orbits of the geodesic flow, hence it is still a Birkhoff section. 


\section{Mixed surfaces and the general cases}\label{S:General}

We now prove Theorem~\ref{T:G1BS} in all remaining cases. 
As in the previous sections, we consider two cases: when the underlying orbifold is a sphere, and when it has positive genus. 
In both cases the Birkhoff section we construct has a vertical part in Section~\ref{S:Vertical} and a horizontal part made of butterfly surfaces as in Section~\ref{S:Horizontal}. 
The novelty consists in a horizontal surface that connects the two parts.

\subsection{Proof of Theorem~\ref{T:G1BS} when $g=0$}\label{S:GeneralSphere}

Fix an integer $n\ge 3$, and $p_1, \dots, p_n\ge 2$ so that $\frac1{p_1}+\dots+\frac1{p_n}<n{-}2$ (this is the condition for the considered orbifold to be hyperbolic). 
In section~\ref{S:Order2} we solved the case $p_1=\dots=p_n=2$ using vertical surfaces. 
In section~\ref{S:Sphere} we solved the case $p_1, \dots, p_n>2$ using butterfly surfaces.
Here we glue both constructions. 
We therefore assume that there is $k$ with $1<k<n$ such that $p_1=\dots=p_k=2$ and $p_{k+1}, \dots, p_n>2$.

\subsubsection{Choice of the metric}\label{S:MetricGeneralSphere}
The choice of the metric is very similar to the choice in~\ref{S:MetricOrder2} and~\ref{S:MetricSphere}, see Figure~\ref{F:CasSphere2}: 
consider a $n$-gon~$E$ in~$\Hy$ with vertices by~$b_1, \dots, b_n$ so that the angle at~$b_i$ is $\pi/p_i$. 
Consider the image~$F$ of $E$ under a reflection across~$(b_1b_2)$, and denote its vertices by $b_1, b_2, b'_3, \dots, b'_n$. 
Identify every side $[b_ib_{i+1}]$ of~$E$ with the side  $[b'_ib'_{i+1}]$ of~$F$. 
The quotient of $E\cup F$ under these identifications is a hyperbolic orientable 2-orbifold that is a sphere with $n$ conic points coming for the vertices~$b_1, \dots, b_n$. 
The total angle at~$b_i$ is~$2\pi/b_i$, so that the order of~$b_i$ is $p_i$. 
Denote it by~$\OO_{0; p_1, \dots, p_n}$. 
The polygons $E$ and $F$ induce two hemispheres on~$\OO_{0; p_1, \dots, p_n}$, while the segments~$[b_1b_2], [b_2b_3],\dots, [b_nb_1]$ form an equator.

\subsubsection{Choice of the boundary orbits}\label{S:BoundaryGeneralSphere}
For~$i=1, \dots, k{-}1$, the edge~$[b_ib_{i+1}]$ connects two conic points of order~$2$. 
Therefore it forms a periodic geodesic on~$\OO_{0; 2, \dots, 2}$ that we denote by~$\beta_i$.
Its lift~$\vrevec \beta_i$ is a single periodic orbit of~$\fgeod$. 

For $i=1, \dots, n$, denote by $r_i$ the rotation around~$b_i$ of angle~$+2\pi/p_i$. 
The group~$G$ generated by the family~$\{r_i\}_{i=1, \dots, n}$ generates the orbifold fundamental group of~$\OO_{0; p_1, \dots, p_n}$, that is, one can identity $\OO_{0; p_1, \dots, p_n}$ with the quotient $\Hy/G$. 
For $i=k, \dots, n$ the composition $r_{i}\circ r_{i+1}^{-1}$ (counting mod~$n$) is a hyperbolic isometry of~$\Hy$ that we denote by~$g_i$ and whose oriented axis we denote by~$\beta_i$ (as in Figure~\ref{F:Ell33} left). 
For $i=k{+}1, \dots, n{-}1$, the projection of~$\beta_i$ on~$\OO_{0;p_1, \dots, p_n}$ is a figure-eight curve (also denoted by~$\beta_i$) that winds positively around~$b_{i+1}$ and negatively around~$b_i$ (as in Figure~\ref{F:Ell33} right).
For $i=k$, the projection of~$\beta_k$ on~$\OO_{0;p_1, \dots, p_n}$ is a simple closed curve (also denoted by~$\beta_k$) that winds positively around~$b_{k+1}$ and~$b_k$, see Figure~\ref{F:ConnectionSphere}.
For $i=n$, the projection of~$\beta_n$ on~$\OO_{0;p_1, \dots, p_n}$ is a simple closed curve (also denoted by~$\beta_n$) that winds negatively around~$b_{1}$ and~$b_n$. 

Denote by~$\vec\beta$ the link~$\vrevec\beta_1\cup\dots\cup\vrevec\beta_{k-1}\cup\vec\beta_{k}\cup\dots\cup\vec\beta_n$ in~$\U\OO_{0; p_1, \dots, p_n}$.

\subsubsection{Choice of the surface}\label{S:SurfaceGeneralSphere}
Color $E$ in white and $F$ in black. 
For~$i=1, 2, \dots, k{-}1$, as in Section~\ref{S:SurfaceOrder2}, choose $s_i$ to be the white side of~$\beta_i$. 
A first part of the constructed surface consists of the surface~$R(\beta_1, s_1)\cup\dots\cup R(\beta_{k-1}, s_{k-1})$. 

For $i=k+1, \dots, n{-}1$, as in Section~\ref{S:SurfaceSphere}, call $m_i$ the double-point on~$\beta_i$. 
Then~$\beta_i$ intersects $\beta_{i-1}$ and $\beta_{i+1}$ twice in a neighbourhood of~$b_i$ (at points called $c_i^E$ and $c_i^F$ depending on the hemisphere in which they lie) and $b_{i+1}$ (at points called $c_{i+1}^E$ and $c_{i+1}^F$), as in Figure~\ref{F:S33} left. 
Therefore they define triangles~$T_i^-$ with vertices $m_i, c_i^E$ and $c_i^F$, and~$T_i^+$ with vertices $m_i, c_{i+1}^E$ and $c_{i+1}^F$. 
Thanks to the orientations of the edges, there is a butterfly surface $B(T_i^-, T_i^+, m_i)$. 
A second part of the constructed section consists of the union of the $n-k-1$ butterfly surfaces~$B(T_{k+1}^-,T_{k+1}^+,m_{k+1})\cup\dots\cup B(T_n^-,T_n^+,m_n)$. 

Finally, the curve~$\beta_n$ intersects~$\beta_{n-1}$ in two points that we call $c_{n}^E$ and $c_{n}^F$, and~$\beta_1$ in one point that we call~$c_1$. 
The curves $\beta_n$ and~$\beta_{n-1}$ delimit a bigon denoted by~$K_1$ (around which they rotate in the clockwise direction) and that contains the order-2 point~$b_1$. 
One fills~$K_1$ with a vector field~$v_1$ that has a singularity at~$b_1$ and is positively tangent to the oriented boundary, see Figure~\ref{F:ConnectionSphere} right. 
The corresponding horizontal surface $H(K_1, v_1)$ has some horizontal boundary in~$\vec\beta_n\cup\vec\beta_{n-1}$ and some vertical boundary in the fibers of the points $c_{n}^E, c_{n}^F$ and~$b_1$. 
It turns out that the vertical boundaries in the fibers of $c_{n}^E, c_{n}^F$ cancel out with the vertical boundary of~$B(T_n^-,T_n^+,m_n)$, while the vertical boundary in the fiber of~$b_1$ cancel out with the vertical boundary of~$R(\beta_1, s_1)$ (the picture in the unit tangent bundle near the fiber of~$b_1$ is the quotient of Figure~\ref{F:VerticalCancel} by an order 2 screw-motion). 

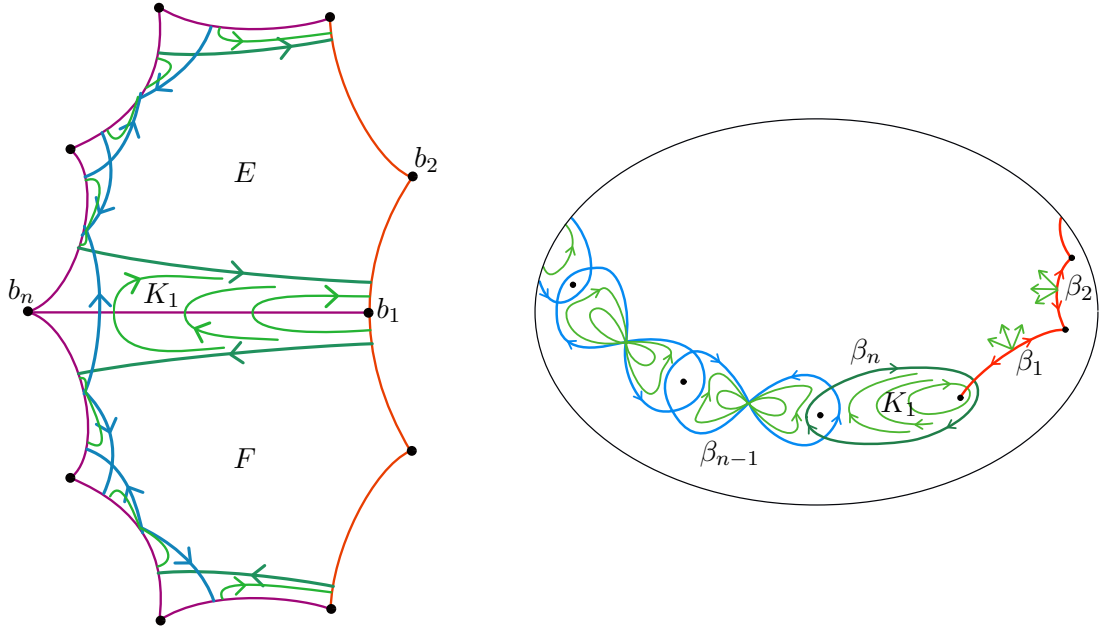
\begin{figure}[ht]
\begin{picture}(160,85)(0,0)
\put(0,-1){\includegraphics[width=.35\textwidth]{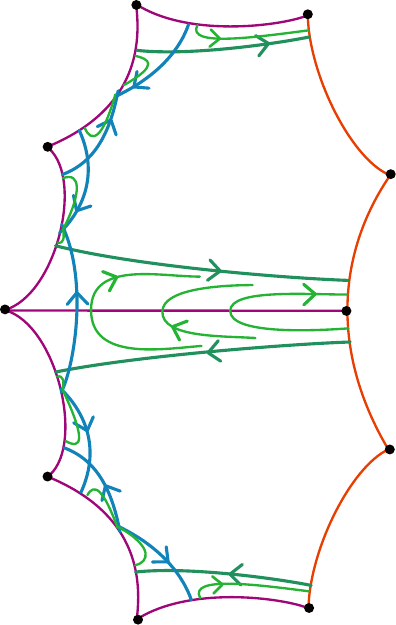}}
\put(68,15){\includegraphics[width=.5\textwidth]{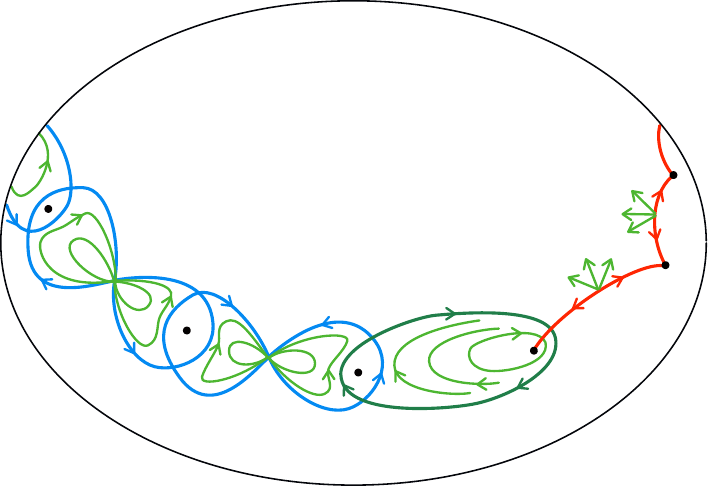}}
\put(28,58){$E$}
\put(28,20){$F$}
\put(47,40){$b_1$}
\put(52,60){$b_2$}
\put(-2,42.5){$b_n$}
\put(16,42){$K_1$}
\put(132,33){$\beta_1$}
\put(138.5,43){$\beta_2$}
\put(110,34.5){$\beta_n$}
\put(114,27.5){$K_1$}
\put(90,21){$\beta_{n-1}$}
\end{picture}
\caption{The surface~$S_{0; 2, 2, \dots, p_{k+1}, \dots, p_n}$.}
\label{F:ConnectionSphere}
\end{figure}

Similarly the curve~$\beta_{k}$ intersects~$\beta_{k+1}$ in two points that we call $c_{k}^E$ and $c_{k}^F$, and~$\beta_{k-1}$ in one point that we call~$c_{k-1}$. 
The curves $\beta_k$ and~$\beta_{k+1}$ delimit a bigon denoted by~$K_k$ (around which they rotate in the trigonometric direction) and that contains the order-2 point~$b_k$. 
One fills~$K_k$ with a vector field~$v_k$ that has a singularity at~$b_k$ and is positively tangent to the oriented boundary. 
The corresponding horizontal surface $H(K_k, v_k)$ has some horizontal boundary in~$\vec\beta_k\cup\vec\beta_{k-1}$ and some vertical boundary in the fibers of the points $c_{k-1}^E, c_{k-1}^F$ and~$b_k$. 
It turns out that the vertical boundaries in the fibers of $c_{k-1}^E, c_{k-1}^F$ cancel out with the vertical boundary of~$B(T_{k+1}^-,T_{k+1}^+,m_{k+1})$, while the vertical boundary in the fiber of~$b_k$ cancel with out the vertical boundary of~$R(\beta_k, s_k)$.  

We then consider the surface~$S_{0; p_1, \dots, p_n}:=R(\beta_1, s_1)\cup\dots\cup R(\beta_{k-1}, s_{k-1})\cup H(K_k, v_k)\cup B(T_{k+1}^-,T_{k+1}^+,m_{k+1})\cup\dots\cup B(T_n^-,T_n^+,m_n) \cup H(K_1, v_1)$. 
By the above discussion, it has no vertical boundary, and its horizontal boundary is exactly the oriented link~$\vec\beta$.

\subsubsection{Computation of the genus}\label{S:GenusGeneralSphere}

We now compute the Euler characteristic of the surface~$S_{0; p_1, \dots, p_n}$. 
As in Section~\ref{S:GenusOrder2}, every rectangle~$R(\beta_i, s_i)$, with $1\le i\le k{-}1$, contributes by~$-1$. 
As in Section~\ref{S:GenusSurface}, every butterfly surface~$B(T_i^-,T_i^+,m_i)$, with $k{+}1\le i\le n{-}1$ contributes by~$-1$. 

What is new compared to the previous sections are the parts $H(K_1, v_1)$ and $H(K_k, v_k)$. 
Both are annuli with two boundary components. 
They share the vertical parts of their boundaries with other surfaces, so that these parts contribute less to the Euler characteristic. 
The total contribution is then~$-1$. 

All-in-all the surface $S_{0; p_1, \dots, p_n}$ is made of $n$ parts that each contribute by~$-1$, so that $\chi(S_{0; p_1, \dots, p_n})=-n$. 
Since it has $n$ boundary components, its genus is~$1$. 

\medskip

The alternative method for computing the genus described in Remark~\ref{R:Genus} also apply here. 
The new contribution to estimate concerns the boundary orbits~$\vec\beta_n$ and $\vec\beta_k$. 
The tangent plane to~$S_{0; p_1, \dots, p_n}$ around these two orbits is vertical only in the points~$c^{E/F}_n$ and~$c^{E/F}_k$ respectively, where the contribution to the self-linking is~$-1/2$. 
Therefore these two boundary components also contribute by~$-1$ the Euler characteristic, hence the genus of~$S_{0; p_1, \dots, p_n}$ is one.

\subsubsection{Intersection with orbits of the geodesic flow}\label{S:IntersectionGeneralSphere}

We are left with proving that the surface~$S_{0; p_1, \dots, p_n}$ we just described intersects all orbits of the geodesic flow~$\fgeod$ on~$\U\Sigma_{0; p_1, \dots, p_n}$.

The reasoning is similar to Sections~\ref{S:IntersectionOrder2} and~\ref{S:IntersectionSphere}: 
by construction, the orbifold~$\OO_{0;p_1, \dots, p_n}$ consists of two hemispheres glued together. 
The collection~$\beta$ lies in a neighborhood of the equator $[b_1b_2]\cup[b_2b_3]\cup\dots\cup[b_nb_1]$. 
We call~$E^\circ$ and $F^\circ$ the two components of $\OO_{0;p_1, \dots, p_n}\setminus\beta$ that do not intersect the equator. 
For $i=k{+}1, \dots, n{-}1$, we call~$B_i$ the bigon containing~$b_i$. 

Consider an arbitrary geodesic~$\gamma$. 
Unless it is in the collection~$\beta$, every time it crosses a triangle of the form $T^-_{i}$, $T^+_{i}$, it then goes in one region of the form $E^\circ, F^\circ$, or $B_i$. 
Therefore $\gamma$ can be decomposed in the form $\dots \cup\gamma_{-1}\cup\gamma_{0}\cup\gamma_{1}\cup\dots$, where each subsegment $\gamma_j$ crosses one of the regions $E^\circ, F^\circ, B_k, \dots, B_n$ and has bounded length. 

One then forms the graph~$G^*_{0; p_1, \dots, p_n}$ as in Section~\ref{S:IntersectionSphere}, see Figure~\ref{F:S3333} top right. 
The geodesic $\gamma$ induces a path~$\gamma^*$ in~$G^*_{0; p_1, \dots, p_n}$, and every time this path uses a bold edge corresponds to a positive intersection of~$\gamma$ with~$S_{0; p_1, \dots, p_n}$. 
Since no path in~$G^*_n$ can avoid the bolded edges, if $\gamma$ is long enough, it necessarily intersects~$S_{0; p_1, \dots, p_n}$, so that the latter is indeed a genus-on Birkhoff section for~$\fgeod$.


\subsection{Proof of Theorem~\ref{T:G1BS} when $g\ge 1$ and $n\le 2g+6$}\label{S:GeneralSurface}

Fix an integer $g\ge 1$. 
Since the case $n=0$ has been treated in Section~\ref{S:Surface}, and since the cases $2g+2\le n\le  2g+6$ have been treated in Sections~\ref{S:Butterfly1} and~\ref{S:Butterfly2}, we fix $n$ with $0<n<2g+2$. 
For simplicity of the presentation we assume $n$ odd. 
The case when $n$ is even will be treated by adding one conic point in the region~$A$ as in Section~\ref{S:Butterfly2}. 
We also fix integers~$p_1, \dots, p_n\ge 2$.

\subsubsection{Choice of a suitable orbifold metric}\label{S:MetricGeneralSurface}
The surface is similar as the one in Section~\ref{S:MetricButterfly}, but with less conic points. 
It is depicted on Figures~\ref{F:SComplete0} and~\ref{F:SComplete1}.

Start with a $2g{+}2$-gon~$A$ in~$\Hy$, whose vertices are denoted by~$a_1, \dots, a_{2g+2}$, so that for $1\le i\le n$ the angle in~$a_i$ is ${\pi}/({2p_i})$, and for $n<i\le 2g+2$  it is $\pi/2$. 
Consider the symmetry~$s_h$ with axis~$(a_1a_2)$ and denote by~$B$ the image of~$A$. 
Also consider the symmetry~$s_v$ with axis~$(a_1a_{2g+2})$ and denote by~$D$ the image of~$A$. 
Finally consider the rotation~$r$ with center $a_1$ and angle $-{\pi}/{p_1}$, and denote by~$C$ the image of~$A$. 

Consider the following identifications of the sides of~$A\cup B\cup C\cup D$:
\begin{itemize}
\item $[a_{2i}a_{2i+1}]$ with $s_{v}([a_{2i}a_{2i+1}])$, 
\item $r([a_{2i}a_{2i+1}])$ with $s_{h}([a_{2i}a_{2i+1}])$, 
\item $[a_{2i+1}a_{2i+2}]$ with $s_{h}([a_{2i+1}a_{2i+2}])$, 
\item $r([a_{2i+1}a_{2i+2}])$ with $s_{v}([a_{2i_1}a_{2i+2}])$. 
\end{itemize}
 In other words, edges of $A$ are identified with their images under $s_{v}$ and $s_h$ in $B$ and $D$ alternatively. 

As in Section~\ref{S:Surface}, the quotient is a hyperbolic orbifold of genus~$g$, that we denote by~$\OO{g; p_1, \dots, p_{n}}$. 
For $1\le i\le n$, since every vertex~$a_i$ is identified with $s_v(a_i), s_h(a_i)$ and $r(a_i)$, in the quotient, the total angle at~$a_i$ is~${2\pi}/{p_i}$, hence $a_i$ is a conic point of order~$p_i$. 
For $n+1\le i\le 2g+2$, the total angle at~$a_i$ is~$2\pi$, hence $a_i$ is a regular point. 
The points $a_1, \dots, a_{n}$ projects onto $n$ conic points of respective orders~$p_1, \dots, p_{n}$. 

Also we denote by $e_i$ the projection of the edge~$[a_ia_{i+1}]$ in the quotient and by $e'_i$ the projection of the edge~$[r(a_i)r(a_{i+1})]$. 
The faces, sides and vertices of~$A\cup B\cup C\cup D$ then induce a graph~$G_g$ on~$\Sigma_g$ with~$4$ faces (that correspond to $A$, $B$, $C$, and $D$), $4g{+}4$ edges (namely $e_1, e'_1, \dots, e_{2g+2}, e'_{2g+2}$), and $2g{+}2$ vertices. 


\subsubsection{Choice of the boundary orbits}\label{S:BoundaryGeneralSurface}

For every $i$ in $\{0, \dots, (n{-}1)/2\}$, consider the oriented arc~$e_{2i+1}$ in~$A$ that starts in $[a_{2i}a_{2i+1}]$ (counting mod $n$) and is orthogonal to it, and ends in~$[a_{2i+2}a_{2i+3}]$ and is orthogonal to it. 
Also consider the oriented arc~$e_{2i}$ in~$A$ that starts in $[a_{2i+1}a_{2i+2}]$ and is orthogonal to it and end in~$[a_{2i-1}a_{2i}]$ and is orthogonal to it. 
These arcs extend into the faces~$B$ and $D$ respectively, and define oriented geodesics $(\alpha_i)_{0\le i\le n}$ on~$\Sigma_{g; p_1, \dots, p_{n}}$.

The same construction in~$C$ (or composing the previous paragraph with the rotation~$r$) yields oriented geodesics $(\alpha'_i)_{1\le i\le n}$ on~$\Sigma_{g; p_1, \dots, p_{n}}$.

Next, for every $i$ in $\{n{+}1, \dots, 2g{+}1\}$, the edges $[a_ia_{i+1}]$ and $[r(a_i)r(a_{i+1})]$ meet in $a_i$ and $a_{i+1}$ with an angle~$\pi$, so that their concatenation is a closed geodesic on~$\Sigma_{g; p_1, \dots, p_n}$, that we denote by~$\alpha_i$. 

\begin{figure}[ht]
\begin{picture}(130,65)(0,0)
\put(00,0){\includegraphics[width=.85\textwidth]{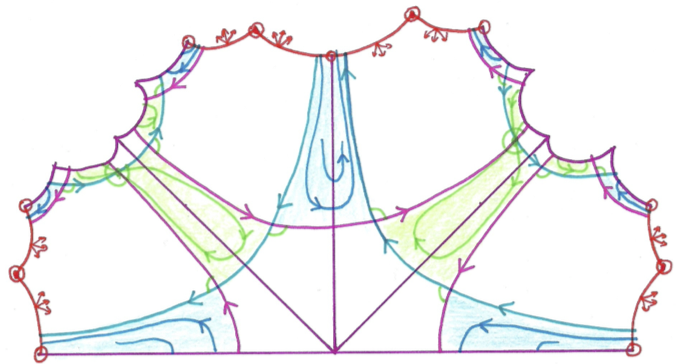}}
\put(80,45){$A^\circ$}
\put(105,15){$B^\circ$}
\put(40,45){$D^\circ$}
\put(20,15){$C^\circ$}
\put(63,-1){$a_1$}
\put(104,43){$a_2$}
\put(90,65){$a_{n+1}$}
\put(58,60){$a_{2g+2}$}
\put(83,35){$\alpha_1$}
\put(91,46){$\alpha_2$}
\put(67,43){$\alpha_0$}
\put(87,55){$\alpha_n$}
\put(94,25){$\alpha'_1$}
\put(105,32){$\alpha_2$}
\put(100,10){$\alpha_0$}
\put(115,26){$\alpha'_n$}
\put(28,23){$\alpha'_1$}
\put(17,31){$\alpha'_2$}
\put(22,10){$\alpha'_0$}
\put(9,26){$\alpha'_n$}
\put(80,65){$\alpha_{n+1}$}
\put(64,63){$\alpha_{2g+1}$}
\put(-3,23){$\alpha_{n+1}$}
\put(-3,7){$\alpha_{2g+1}$}
\put(37,65){$\alpha_{n+1}$}
\put(56,15){$Q_1$}
\put(57,35){$K_0$}
\put(20,-1.5){$K_0$}
\put(100,-1){$K_0$}
\put(80,25){$K_1$}
\put(40,25){$K'_1$}
\end{picture}
\caption{A fundamental domain for the orbifold $\Sigma_{g; p_1, \dots, p_{n}}$, here with $g=2, n=3$. 
The geodesics $\alpha_0,\alpha_2, \dots,\alpha_{n-1}$ and $\alpha'_0$, $\alpha'_2,\dots,\alpha'_{n-1}$ are shown in light blue. 
The geodesics $\alpha_1,\alpha_3, \dots,\alpha_{n}$ and $\alpha'_1,\alpha'_1,\dots,\alpha'_{n}$ are shown in pink. 
The geodesics $\alpha_{n+1},\dots, \alpha_{2g+1}$ are shown in orange. 
Together they form a tesselations into regions $A^\circ,  B^\circ, C^\circ, D^\circ, (Q_i)_{1\le i\le n}, (K_i)_{0\le i\le n}, (K'_i)_{0< i< n}$. 
The surface~$S_{g; p_1, \dots, p_n}$ is the union of the (green) butterfly surfaces~$B(K_1, K_2, s_1), B(K'_1, K'_2, s'_1), B(K_3, K_4, s_3), \dots, B(K'_{n-2}, K'_{n-1}, s'_{n-2})$, the (orange) vertical surfaces~$R(\alpha_{n+1}, s_{n+1}), \dots, R(\alpha_{2g+1}, s_{2g+1})$, connected by the two (blue) horizontal surfaces $H(K_0, v_0)$ and $H(K_n, v_n)$.
}
\label{F:SComplete0}
\end{figure}
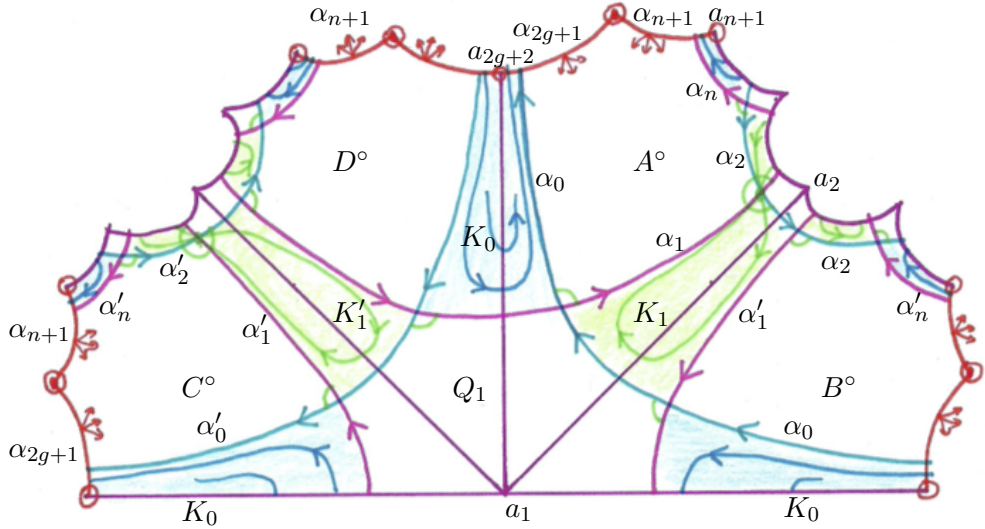

\begin{figure}[ht]
\begin{picture}(130,75)(0,0)
\put(00,0){\includegraphics[width=.85\textwidth]{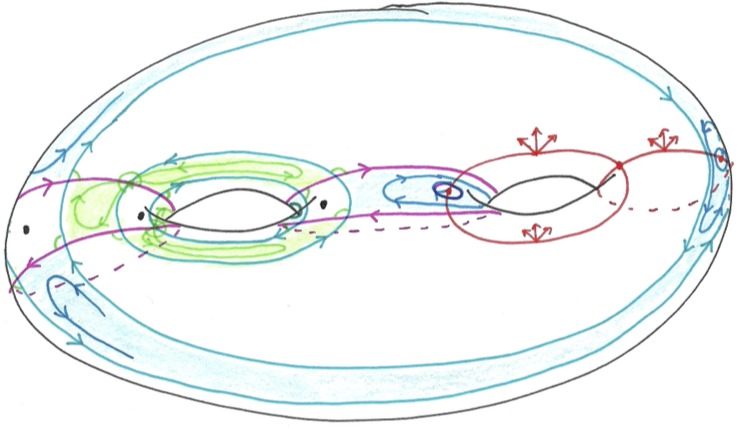}}
\end{picture}
\caption{Another view of $\Sigma_{g; p_1, \dots, p_{n}}$, after identifications of the sides of~$A, B, C, D$, with the same colors as in Figure~\ref{F:SComplete0}.}
\label{F:SComplete1}
\end{figure}

Finally consider the link~$L_{g; p_1, \dots, p_{n}}:=\vec\alpha_0\cup\vec\alpha_1\cup\dots\cup\vec\alpha_{n}\cup\vec\alpha'_0\cup\vec\alpha'_1\cup\dots\cup\vec\alpha'_{n}\cup\vrevec\alpha_{n+1}\cup\dots\cup \vrevec\alpha_{2g+1}$. 

\subsubsection{Choice of the surface}\label{S:SurfaceGeneralSurface}

The collection~$(\alpha_i)_{0\le i\le n}\cup(\alpha'_i)_{0\le i\le n}$ $\cup (\alpha_i)_{{n+1}\le i\le {2g+1}}$ decomposes~$\OO_{g; p_1, \dots, p_{n}}$ into $3n+6$ regions, namely $n$ quadrilaterals containing the points $a_1, \dots, a_{n}$ that we denote by~$Q_i$, $2n+2$ quadrilaterals containing the central parts of the edges $[a_ia_{i+1}]$ and $[r(a_i)r(a_{i+1})]$ (with $0\le i\le n$ and counting mod $2g+2$) that we denote by $K_i$ and $K'_i$ respectively, and four faces corresponding the centers of $A, B, C,$ and~$D$ that we denote by~$A^\circ, B^\circ, C^\circ, D^\circ$. 

Firstly, for $1\le i\le (n{-}1)/2$, the quadrilaterals $K_{2i-1}$ and $K_{2i}$ meet at a single vertex that we call~$s_{2i-1}$, and define a butterfly as defined in Section~\ref{S:Butterfly}. 
We then consider the butterfly surface~$B(K_{2i-1}, K_{2i}, s_{2i-1})$. 
Similarly, $K'_{2i-1}$ and $K'_{2i}$ meet at a single vertex that we call~$s'_{2i-1}$, and define a butterfly. 
We then consider the butterfly surface~$B(K'_{2i-1}, K'_{2i}, s'_{2i-1})$. 

Secondly, for $n{+}1\le i\le 2g{+}1$, denote by $e_i$ the edge~$[a_{i}a_{i+1}]$, and denote by~$s_i$ the side of~$A$. 
Similarly, denote by $e'_i$ the edge~$[r(a_{i})r(a_{i+1})]$, and denote by~$s'_i$ the side of~$C$. 
We the consider the vertical surfaces~$R(e_{i}, s_{i})$ and $R(e'_{i}, s'_{i})$. 

Thirdly, the region~$K_0$ consists of a quadrilateral bounded by $\alpha_0, \alpha_1, \alpha'_0, \alpha'_1$ and containing~$a_{2g+2}$ in its center. 
We consider a vector field~$v_0$ that rotates in the trigonometric direction around~$a_{2g+2}$, is tangent to the sides of~$K_0$, and is everywhere curved. 
We then consider the horizontal surface~$H(K_0, v_0)$. 
It is topologically an annulus with one boundary component which is the fiber of~$a_{2g+2}$, and the other component has four horizontal arcs in the lifts~$\vec\alpha_0, \vec\alpha_1, \vec\alpha'_0, \vec\alpha'_1$, and for vertical arcs at the vertices of~$K_0$. 
It turns out that the fiber~$\U a_{2g+2}$ is exactly one of the vertical boundary components of~$R(e_{2g+1}, s_{2g+1})\cup R(e'_{2g+1}, s'_{2g+1})$, see Figure~\ref{F:VerticalCancel}. 
Also, the vertical boundary components in the vertices of~$K_0$ match with some vertical boundary components of the butterfly surfaces~$B(K_1, K_2, s_1)$ and $B(K'_1, K'_2, s'_1)$. 

\begin{figure}[ht]
\includegraphics[width=.5\textwidth]{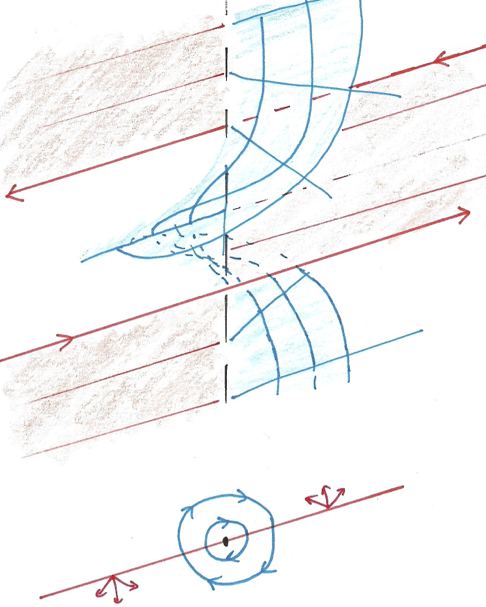}
\caption{How the blue surface~$H(K_0, v_0)$ connects with the orange surfaces $R(e_1, s_1), R(e'_1, s'_1)$ in~$\U a_{2g+2}$. 
In particular they are disjoint and their vertical boundaries match. 
The union is not smooth, but can easily be smoothed by staying transverse to the geodesic flow. 
}
\label{F:VerticalCancel}
\end{figure}

Similarly the region~$K_n$ consists of a quadrilateral bounded by~$\alpha_{n-1}, \alpha_{n}, \alpha'_{n-1}, \alpha'_{n}$ and containing~$a_{n+1}$ in its center. 
We consider a vector field~$v_n$ that rotates in the clockwise direction around~$a_n$, is tangent to the sides of~$K_n$, and is everywhere curved. 
We then consider the horizontal surface~$H(K_n, v_n)$. 
It is topologically an annulus. 
One boundary component is the fiber~$\U a_{n}$. It matches is exactly one of the vertical boundary components of~$R(e_{n+1}, s_{n+1})\cup R(e'_{n+1}, s'_{n+1})$. 
Also the vertical parts of the other boundary components match with some vertical boundary components of the butterfly surfaces~$B(K_{n-2}, K_{n-1}, s_{n-2})$ and $B(K'_{n-2}, K'_{n-1}, s'_{n-2})$. 

All-in-all we consider the surface defined as the union 
\begin{eqnarray*}
S_{g; p_1, \dots, p_n}&:=& B(K_{1}, K_{2}, s_{1})\cup B(K'_{1}, K'_{2}, s'_{1})\cup B(K_{3}, K_{4}, s_{3})\cup\dots\cup B(K'_{n-2}, K'_{n-1}, s'_{n-2})\\
&&\cup H(K_n, v_n)\\ 
&&\cup R(e_{n+1}, s_{n+1})\cup R(e'_{n+1}, s'_{n+1})\cup R(e_{n+2}, s_{n+2})\cup\dots\cup R(e'_{2g+1}, s'_{2g+1})\\
&&\cup H(K_0, v_0)\ .
\end{eqnarray*}

By the previous discussion (and in particular Figure~\ref{F:VerticalCancel}), all the vertical parts of the boundary two-by-two cancel, so that is has only horizontal boundary, and this (oriented) boundary is exactly~$-L_{g; p_1, \dots, p_n}$.

\subsubsection{Computation of the genus}\label{S:GenusGeneralSurface}

The surface~$S_{g; p_1, \dots, p_{n}}$ is made of several parts. 
We compute its Euler characteristic by computing the contribution of each part. 
As previously, when a vertex or an edge lies at the intersection of two or more parts, its contribution is equidistributed between those parts. 

The first part consists of $n-1$ butterflies (in green on the pictures). 
Each butterfly is made of two wings that are the lifts of vector fields on quadrilaterals. 
As in Section~\ref{S:GenusSphere}, each such wing contributes by~$-1$ to the Euler characteristic, so that 
this part contributes by~$-2(n-1)$. 

The second part consists of~$2(2g+1-n)$ vertical rectangles (in orange on the pictures). 
As in Section~\ref{S:GenusSurface}, each rectangle contributes by $-1$ to the Euler characteristic, so that this part contributes by~$-2(2g+1-n)$. 

The third part consists of the two gluing horizontal surfaces~$H(K_0, v_0)$ and $H(K_n, v_n)$ (in blue on the pictures). 
Each of them is an annulus. 
In particular the boundary in the fiber contributes by~$0$, the boundary that is the lift of the boundary of~$K_0$ ({\it resp.} $K_n$) contributes by~$-2$, and the face contributes by~$0$ since it is an annulus. 
Therefore these two gluing parts contribute by~$-4$. 

All-in-all, we have $\chi(S_{g; p_1, \dots, p_n})=-2(n-1)-2(2g+1-n)-4=-4g-4$.
Since the boundary of this surface is the link~$L_{g; p_1, p_n}$, which has $4g+4$ components, its genus is one.

\begin{remark}
This computation of the genus is the most crucial and delicate part of the construction (and of the proof). 
In the first version of the article, there was a mistake, which led to the construction of surface of genus~$2$ instead of a torus. 
Despite many attempts, we are unable to fix this gap, hence our construction only works with the restriction~$n\le 2g+6$. 
\end{remark}

\begin{remark}\label{R:AlternativeGenus}
As in Remark~\ref{R:Genus}, we provide another way to compute the genus of~$S_{g; p_1, \dots, p_n}$, which we recap here since this is the crucial part of the paper. 

By construction, $S_{g; p_1, \dots, p_n}$ is transverse to the geodesic flow~$\fgeod$. 
The latter is of Anosov type, so that the two non-singular 2-dimensional foliations~$\Fs, \Fu$ are tangent to it and invariant. 
The intersection~$S_{g; p_1, \dots, p_n}\cap\Fs$ is therefore a 1-dimensional foliation on~$S_{g; p_1, \dots, p_n}$ that is non-singular in the interior of the surface. 
However there are singularities on the boundary components of~$S_{g; p_1, \dots, p_n}$. 
These singularities correspond to tangencies between the tangent plane to~$S_{g; p_1, \dots, p_n}$ on a boundary component and the tangent plane to~$\Fs$. 
They are of index~$-1/2$. 

Given a boundary component~$\gamma$ of~$S_{g; p_1, \dots, p_n}\cap\Fs$ (which is a periodic orbit of~$\fgeod$), 
define its self-linking~$\slk^{\Fs, S_{g; p_1, \dots, p_n}}(\gamma)$ relatively to~$\Fs$ and~$S_{g; p_1, \dots, p_n}$ as the algebraic intersection number of~$S_{g; p_1, \dots, p_n}$ with a copy of~$\gamma$ pushed in the direction of~$\Fs$. 
By the Poincaré-Hopf formula, we have $\chi(S_{g; p_1, \dots, p_n}) = -\sum_{\gamma\in\partial S_{g; p_1, \dots, p_n}} |\slk^{\Fs, S_{g; p_1, \dots, p_n}}(\gamma)|$. 

Notice that, if~$S_{g; p_1, \dots, p_n}$ is indeed a Birkhoff section (as will be proved in the next section), the self-linking number $\slk^{\Fs, S_{g; p_1, \dots, p_n}}(\gamma)$ does not vanish. 
Therefore $S_{g; p_1, \dots, p_n}$ has genus 1 if and only if for every boundary component~$\gamma$ one has $\slk^{\Fs, S_{g; p_1, \dots, p_n}}(\gamma)=\pm1$. 

Back to our construction, we note that the normal direction~$\Fs$ to a periodic orbit is isotopic to the~$\U$-direction of the fiber, which is easier to visualize. 
Therefore we have to check that the for every boundary component~$\gamma$, the surface~$S_{g; p_1, \dots, p_n}$ is positively tangent to the fiber direction exactly once. 

For an orbit of type~$\vec\alpha_i$ with $1\le i\le n{-}1$, the surface~$S_{g; p_1, \dots, p_n}$ is vertical only when $\alpha_i$ intersects another orbit~$\alpha_{i\pm1}$ or~$\alpha'_{i\pm1}$. 
At such an intersection, the surface makes half a turn in the vertical direction. 
However this half turn is in the negative direction when the intersection point in at the extremity of a wing of a butterfly, and in the positive direction when the intersection corresponds to the body of the butterfly. 
Since every orbit of the type contains 3 intersections of the first type and 1 of the second type, its self-linking is~$-1$. 

The same happens for an orbit of type~$\vec\alpha'_i$ with $1\le i\le n$. 

For an orbit of type~$\vec{\alpha_i}$ with $n{+}1\le i\le 2g{+}1$, the surface~$S_{g; p_1, \dots, p_n}$ makes a half-turn in the negative direction every time~$\alpha_i$ intersects $\alpha_{i\pm 1}$, see Figure~\ref{F:Rec} right. 
Since there are two such intersection point, the self-linking is also~$-1$. 

The remain the cases of~$\vec\alpha_0$ and~$\vec\alpha_n$ which are similar. 
By the same argument as for~$\vec\alpha_1 , \dots,$ $\vec \alpha_{n-1}$, the surface makes two negative half-turns around~$\vec\alpha_0$ where~$\alpha_0$ intersects~$\alpha_1$. 
What is different is that where~$\alpha_0$ intersects~$\alpha_{2g+1}$ and~$\alpha'_{2g+1}$, there is no half-turn. 
Therefore one also find that the self-linking numbers of~$\vec\alpha_0$ and~$\vec\alpha_n$ are~$-1$. 

All-in-all, all boundary orbits have self-linking~$-1$, so that $S_{g; p_1, \dots, p_n}$ is indeed a punctured torus. 
\end{remark}

\subsubsection{Intersection with orbits of the geodesic flow}\label{S:IntersectionGeneralSurface}

We are left with proving that the surface $S_{g; p_1, \dots, p_{n}}$ is a Birkhoff section for the geodesic flow on $\U\Sigma_{g; p_1, \dots, p_{n}}$. 
By construction, its interior is transverse to the geodesic flow, and its boundary is tangent to it. 
In order to prove that it intersects all orbits of~$\fgeod$ in bounded time, we use the same type of argument as in Sections~\ref{S:IntersectionSurface}, \ref{S:IntersectionOrder2}, \ref{S:IntersectionSphere}, \ref{S:IntersectionButterfly}, and  \ref{S:IntersectionGeneralSphere}

Consider the graph~$G^*_{g; p_1, \dots, p_n}$ whose vertices are denoted by $A^*, B^*, C^*, D^*, Q^*_1, Q^*_2, \dots,$ $Q^*_n$. 
There is an oriented edge from a vertex $U^*$ to a vertex~$V^*$ everytime the corresponding faces~$U$ and $V$ share an edge (necessarily of type~$e_i$ or $e'_i$ with $n{+}1\le i\le 2g{+}1$) or are separated by a face of type~$K_i$ or~$K'_i
$ (with $1\le i\le n{-}1$). 
The graph we obtain is similar as the one in Figure~\ref{F:Sg333Dual}.

Now consider an oriented geodesic arc~$\gamma$ on~$\Sigma_{g; p_1, \dots, p_{n}}$. 
It intersects faces of type $Q_i, K_i, K'_i$, $A^\circ, B^\circ, C^\circ, D^\circ$. 
Given the adjacencies between those polygons, it defines a path~$\gamma^*$ in the dual graph~$G_{g;p_1, \dots, p_n}^*$. 
When $\gamma^*$ follows a bold edge in this graph, the lift~$\vec\gamma$ intersects once (positively) the surface~$S_{g; p_1, \dots, p_{n}}$. 
Since every path in~$G_{g; p_1, \dots, p_{n}}^*$ contains bold edges (in fact at least every third edge), every long enough arc of orbit of the geodesic flow in~$\U\Sigma_{g; p_1, \dots, p_{n}}$ intersects~$S_{g; p_1, \dots, p_{n}}$ positively.


\bibliographystyle{siam}

\end{document}